\documentclass[10pt,reqno]{amsart}

\usepackage{amsmath,amsfonts,amssymb,latexsym,dsfont,mathrsfs,euscript}
\usepackage{graphicx}
\usepackage{graphics,epsfig,pictexwd,dcpic}
\usepackage{mathtools,color}
\usepackage{caption}
\usepackage{subcaption}

\renewcommand{\aa}{\mathbb{A}}
\newcommand{\pp}{\mathbb{P}}
\newcommand{\zz}{\mathbb{Z}}
\newcommand{\qq}{\mathbb{Q}}

\renewcommand{\tt}{\mathbb{T}}

\newcommand{\emb}{\hookrightarrow}
\newcommand{\surj}{\twoheadrightarrow}
\newcommand{\isom}{\xrightarrow{_\sim}}

\newcommand{\im}{\operatorname{im}}

\renewcommand{\hom}{\operatorname{Hom}}

\newcommand{\pr}{\operatorname{pr}}

\DeclareMathOperator{\ass}{Ass}

\newcommand{\bc}[1]{\tilde{#1} }
\newcommand{\bcw}[1]{\widetilde{#1} }
\newcommand{\bcpb}[1]{\tilde{#1}^{\ast} }
\newcommand{\bcwpb}[1]{\widetilde{#1}^{\ast} }

\DeclareMathOperator{\spec}{Spec}

\DeclareMathOperator{\blow}{Bl}
\DeclareMathOperator{\supp}{Supp}

\DeclareMathOperator{\trop}{Trop}
\DeclareMathOperator{\trpz}{\mathit{val}}

\newcommand{\brk}[1]{\left( #1 \right) }
\newcommand{\crb}[1]{\left\{ #1 \right\} }

\newcommand{\inn}[1]{\left\langle #1 \right\rangle }

\newcommand{\cl}[1]{\mathfrak{#1}}
\newcommand{\cls}[1]{\mathcal{#1}}

\newcommand{\cla}[1]{\EuScript{#1}}

\newcommand{\twovec}[2]{\left (\hspace{-2pt}\begin{array}{c} #1 \\ #2 \\ \end{array}\hspace{-2pt} \right )}

\newcommand{\twomat}[4]{\left (\begin{array}{cc} #1 & #2 \\ #3 & #4 \\ \end{array}\right )}

\newcommand{\gl}[1]{\mathrm{GL}_{#1}}
\newcommand{\gln}{\mathrm{GL}_n}
\renewcommand{\sl}[1]{\mathrm{SL}_{#1}}
\newcommand{\sln}{\mathrm{SL}_n}
\newcommand{\so}[1]{\mathrm{SO}_{#1}}

\newcommand{\pgln}{\mathrm{PGL}_n}

\renewcommand{\labelenumi}{(\roman{enumi})}

\newtheorem{thm}{Theorem}[section]
\newtheorem{conj}{Conjecture}[section]
\newtheorem{prop}[thm]{Proposition}
\newtheorem{lem}[thm]{Lemma}
\newtheorem{cor}[thm]{Corollary}

\newtheorem{defn}[thm]{Definition}
\newtheorem{exam}[thm]{Example}
\newtheorem{rem}[thm]{Remark}

\begin{document}

\title{Spherical Tropicalization}
\author{Jenia Tevelev}
\author{Tassos Vogiannou}
\date{\today}

\begin{abstract}
We extend tropicalization and tropical compactification of subvarieties of algebraic tori 
to subvarieties of spherical homogeneous spaces. 
Given a tropical compactification of a subvariety, we show that the support of the colored fan of the ambient spherical variety agrees with the tropicalization of the  subvariety. 
The proof is based on our equivariant version of the flattening by blow-up theorem.
We provide many examples.
\end{abstract}

\dedicatory{In memory of Ernest Borisovich Vinberg}

\maketitle


\section{Introduction} \label{intro}

Let $k$ be an algebraically closed field, 
$K=k((t))$ the field of Laurent series 
and $\overline{K}=\bigcup_n k((t^{1/n}))$ the field of Puiseux series 
with valuation 
$$
\nu:\overline{K}^{\times}\rightarrow \qq,\quad \sum_n c_nt^n\mapsto \min\{n:c_n\neq 0\}.
$$

Let $\tt^n= (k^{\times})^n$ be the algebraic torus, 
$\Lambda=\hom(\tt^n,k^{\times})$ its character group, and 
$\mathrm{Q}=\hom(\Lambda,k^{\times})\cong \qq^n$. The valuation $\nu$ induces a surjective map:
\begin{equation}\label{torictrop}
\trpz:T(\overline{K})\rightarrow \mathrm{Q}
\end{equation}
that sends $(x_1(t),\ldots,x_n(t))\in (\overline{K}^{\times})^n$ to $(\nu(x_1(t)),\ldots,\nu(x_n(t)))\in \qq^n$.
Given a closed subvariety $Y\subseteq \tt^n$, its tropicalization 
$\trop Y$ is the image $\trpz( Y(\overline{K}))\subseteq\qq^n$. It is a piece-wise linear set,
more precisely the support of a rational polyhedral fan. 
More about tropicalizations and their use can be found in 
\cite{MB}, \cite{G}, \cite{EKL}.

Given a closed subvariety $Y\subseteq \tt^n$, a tropical compactification $\overline Y$, introduced in~\cite{Te}, 
is the closure of $Y$ in a toric variety $X$ of $\tt^n$ such that 
$\overline{Y}$ is a complete variety 
and the multiplication map $\mu_{\overline{Y}}$ of $\overline{Y}$ is faithfully flat:
\begin{equation}\label{sgasrgasrgar}
\mu_{\overline{Y}}:T\times \overline{Y}\rightarrow X, \quad (g,x)\mapsto gx.
\end{equation}
The toric variety $X$ is  
given by a fan $\cla{F}$ in $\mathrm{Q}$.
The relation between the tropicalization
and a tropical compactification is very simple, $\supp \cla{F}=\trop Y$.
Existence of tropical compactifications
was shown in \cite{Te}, see also \cite{HKT, ST, LQ, U}.

We extend the tropicalization and tropical compactifications to subvarieties $Y\subset G/H$ of spherical homogeneous spaces of connected reductive groups. Spherical means that
a Borel subgroup $B\subset G$ acts on $G/H$ with an open orbit. 
As in the toric case, spherical varieties, i.e.~equivariant open embeddings of $G/H$
into normal $G$-varieties $X$, are in a bijection with combinatorial data (colored fans). 
This correspondence is reviewed briefly in \S\ref{spherreview},
see also \cite{LV, K, Ti}. 
There is a map 
\begin{equation}\label{spherictrop}
\trpz: (G/H)(\overline{K})\rightarrow \mathrm{Q}
\end{equation}
analogous to (\ref{torictrop}),
where $\mathrm{Q}=\hom(\Lambda,\qq)$ 
and $\Lambda$ is the weight lattice, i.e.~the subgroup of characters of $B$ that are weights of $B$-semi-invariant functions on $G/H$. The image of $\trpz$ is the valuation cone $\cls{V}$. 
We define the spherical tropicalization 
$\trop Y$ as $\trpz(Y(\overline{K}))$. By Lemma \ref{workink},
$\trop Y$ is a conical set generated by $\trpz(Y(K))$.

\begin{defn}\label{rgaerhaer}
\renewcommand{\labelenumi}{(\roman{enumi})}
The closure $\overline{Y}\subseteq X$ is called a \emph{tropical compactification} of $Y$ if $\overline{Y}$ is complete, and the multiplication map $\mu_{\overline{Y}}:\,G\times \overline{Y}\rightarrow X$ is faithfully flat.
\end{defn}


\begin{thm} \label{mainthmspherical}
For any closed subvariety $Y$ of a spherical homogeneous space $G/H$,
\begin{enumerate}
\item Tropical compactifications of $Y$ in toroidal spherical varieties exist.
\item If $\overline{Y}\subseteq X$ is a tropical compactification, where $X$ is a spherical variety associated to a colored fan $\cla{F}$, then $\supp \cla{F}=\trop Y$.
\end{enumerate}
\end{thm}

Definition~\ref{rgaerhaer} makes sense for subvarieties of arbitrary homogenous spaces
and we prove Theorem~\ref{mainthmspherical} in \S\ref{tropcompactif} based on 
analogous results in this more general setting.
This, in turn, is based
on our generalization (Theorem~\ref{equivflat}) of the flattening by blow-up theorem \cite{RG}
to the equivariant setting.

In the toric case, tropicalization is usually defined in a more general setting, for subvarieties of $T$ defined over the field 
$\overline K$ of Puiseux series, in which case it is a support of a polyhedral complex rather than a fan.
Tropical compactifications in this setting have been studied in \cite{LQ,G,HKT}.
In the spherical setting, the tropicalization of a subvariety $Y\subset G/H$ defined over $\overline K$
is also well-defined, using the formula $\trpz(Y(\overline{K}))$. It would be interesting
to study the tropical compactifications in this setting and to exhibit interesting examples.

We work out many examples of spherical tropicalization
and spherical tropical compactifications for subvarieties of non-toric spherical homogeneous spaces
in \S\ref{allexamples}.
A basic example is $\gln$ viewed as a spherical homogeneous space of $\gln\times \gln$, which
acts on $\gln$ by left and right multiplication. Recall that if $x=(x_{ij}(t))$ is an invertible matrix with entries in $K$, there exist matrices $g=(g_{ij})$ and $h=(h_{ij})$ with entries in $k[[t]]$, such that $gxh$ is in ``reversed'' Smith normal form
$$
gxh= \brk{
\begin{array}{cccc}
t^{\alpha_1} & 0 & \ldots & 0 \\
0 & t^{\alpha_{2}} & \ldots & 0 \\
\vdots & \vdots & \ddots & \vdots \\
0 & 0 & \ldots & t^{\alpha_n} \\
\end{array}},
$$
for some integers $\alpha_1\geq \cdots \geq \alpha_n$,
which we call the \emph{invariant factors} of $x$ slightly abusing the standard terminology. 


\begin{thm} \label{tropgln}
Let $Y\subset \gln$ be a closed subvariety. 
Then $\trop Y\subset \qq^n$ is the conical subset generated by 
$n$-tuples 
of invariant factors 
of   matrices $x\in Y(K)$. 
\end{thm}

If a closed subvariety $Y\subset \gl{n}$ admits a parametrization then $\trop Y$ can be calculated in a straightforward and elementary way
analogous to the approach~\cite{ST} in the toric case.
 For instance, to find the tropicalization of the subvariety $Y=V(x_{11}-x_{22},x_{12}^3-x_{21})\subset \gl{2}$, 
one can write a matrix in $Y(K)$ in the form
$$
\twomat{y(t)}{z(t)}{z(t)^3}{y(t)},\quad y(t),z(t)\in K,
$$
which allows to easily find its possible invariant factors. The tropicalization of this variety is a striped area
 in Figure \ref{tropexam5}. The gray area is the rest of the valuation cone.

\begin{figure}
\caption{Tropicalization of $Y=V(x_{11}-x_{22},x_{12}^3-x_{21})$}
\label{tropexam5}
\centering
\includegraphics[width=0.45\textwidth]{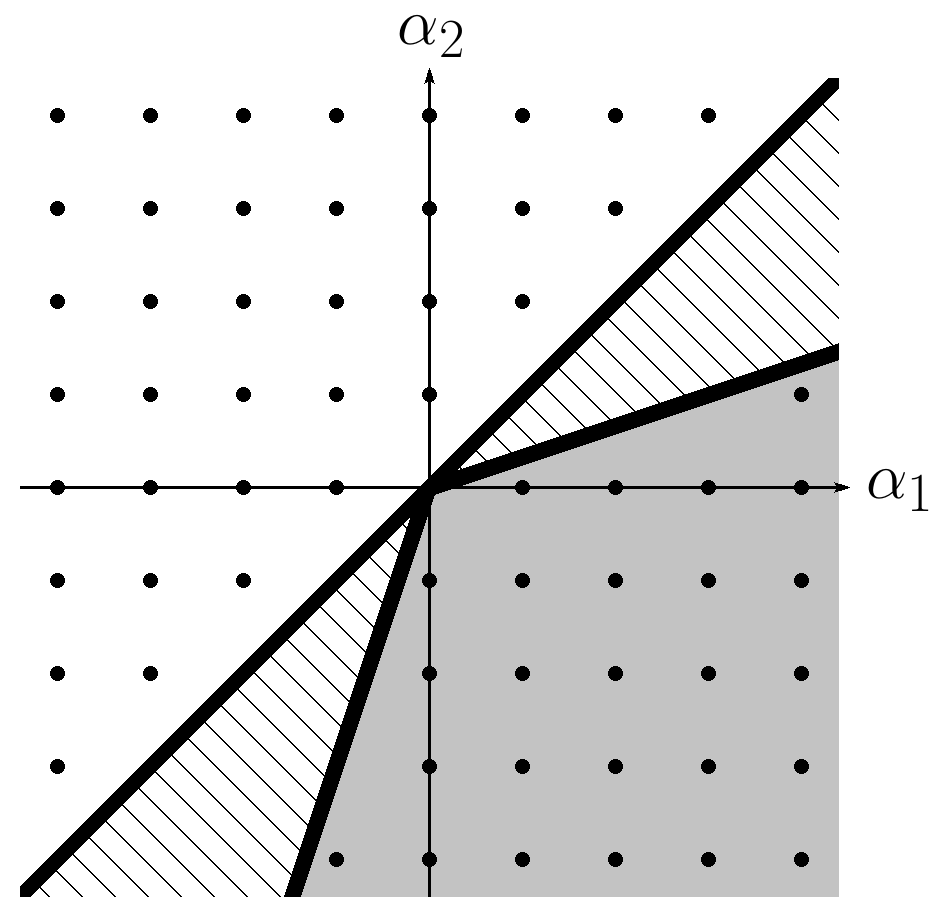}
\end{figure}

Valuations of defining equations of $Y$ impose restrictions on possible invariant factors of matrices in $Y(K)$,
e.g.~if $Y\subset\gl{2}$ is given by $x_{11}^2x_{12}-x_{22}^5+x_{11}x_{21}^3=1$
then any matrix $(x_{ij}(t))\in Y(K)$ satisfies
$$
\min\crb{2\nu(x_{11}(t))+\nu(x_{12}(t)),-5\nu(x_{22}(t)),\nu(x_{11}(t))+3\nu(x_{21}(t))}=0.
$$
But, as in the toric case, given a set of defining equations $f_1,\ldots,f_n$ of $Y$, 
$\trop Y$ is not always the intersection of $\trop V(f_i)$ (see Example \ref{examplesgl2}).

Many interesting moduli spaces in algebraic geometry are subvarieties of spherical varieties,
for example the representation 
variety $Y$ of the fundamental group of the $2$-sphere with $3$ punctures
is a subvariety of $G=(\gl{n})^3$, which can viewed as a spherical homogeneous space of $G\times G$.
In \S\ref{represvarty} we show that the tropicalization of 
$Y$ is a cone given 
by the Horn inequalities. 
We plan to compute the tropical compactification of $Y$ in a future work.
Spherical tropicalization and tropical compactification of arbitrary representation varieties
is a very interesting problem, which can lead to a 
compactification theory for character varieties, see~\cite{V}.

\smallskip
\noindent \emph{Acknowledgements}. 
The project was  supported by the NSF grant DMS-1701704, Simons Fellowship, and the HSE University Basic Research Program and Russian Academic Excellence Project '5-100'.

%
%
%


\section{Compactifications of Subvarieties of General Homogeneous Spaces} \label{tropcompactifgeneral}

\subsection{Equivariant Flattening by Blow-up} \label{equivariantflatsect}

In this subsection all schemes and morphisms are over a fixed noetherian scheme $S$ 
including a flat affine surjective group scheme ~$G$ of finite type, 
a  $G$-morphism $f:\,{X}\rightarrow {Y}$ of $G$-schemes of finite type,
a dense open $G$-subset $U\subset Y$, and 
a closed $G$-subscheme $Z\subset X$ or more generally a coherent equivariant sheaf (also called $G$-sheaf) $\cla{M}$ on $X$.
We write
\begin{equation}\label{argaerh}
\mu:G\times_S X\rightarrow X,\quad (g,x)\mapsto gx
\end{equation}
for the multiplication map. 
We will assume that 
$f|_{Z\cap f^{-1}(U)}$ is a flat morphism, or more generally that 
the coherent sheaf  $\cla{M}|_{f^{-1}(U)}$ is flat over~$U$\footnote{
When $Y$ is reduced, there exists an open set $U'$  such that $\cla{M}|_{f^{-1}(U')}$ is flat \cite[Th.~6.9.1]{EGAIV}
over $U'$. 
The image $U=\mu(G\times_S U')$ is then a $G$-stable open set such that $\cla{M}|_{f^{-1}(U)}$ is flat over~$U$.}.

\begin{defn}[{\cite[Ch. 4]{R}}]\label{stricttransformchar}\label{puretransformclosed}
Consider a Cartesian diagram
\begin{equation}
\begindc{\commdiag}[40]
\obj(0,8)[xp]{$\bcw{X}$} \obj(0,0)[yp]{$\bcw{Y}$} \obj(16,0)[y]{$Y$} \obj(16,8)[x]{$X$}
\mor{xp}{yp}{$\bc{f}$}[-1,0] \mor{xp}{x}{$\bc{u}$} \mor{yp}{y}{$u$} \mor{x}{y}{$f$}
\enddc
\label{sefvwerv}\end{equation}
where 
$u$
is a projective $G$-morphism, which induces an isomorphism of open dense 
$G$-subsets $u|_{\bcw{U}}:\bcw{U}\isom U$. 
Let $\bcw{\cla{M}}=\bcpb{u}\cla{M}$.
A quotient coherent sheaf $\bcw{\cla{M}}^{pt}=\bcw{\cla{M}}/\bcw{\cla{N}}$  is called a pure transform
of $\cla{M}$ if $\bcw{\cla{N}}$ vanishes on $\bc{f}^{-1}(\bcw{U})$ and
$\ass(\bcw{\cla{M}}^{pt})\subseteq \bc{f}^{-1}(\bcw{U})$.
Equivalently, $\bcw{\cla{N}}\subset \bcw{\cla{M}}$ is  the subsheaf of all sections supported on $\bc{f}^{-1}(\bcw{Y}-\bcw{U})$.
For a subscheme $Z\subset X$, the pure transform is defined as the scheme-theoretic closure
$$\bcw{Z}=\overline{\bc{u}^{-1}(Z\cap f^{-1}(U)) }\subseteq\bcw{X}.$$
\end{defn}

Pure transforms of coherent sheaves and subschemes are related as follows:

\begin{lem} \label{puretransagree}
$(\cla{O}_X/\cla{I}_Z)^{pt}=\cla{O}_{\bcw{X}}/\cla{I}_{\bcw{Z}}$, where 
$\cla{I}_Z$, $\cla{I}_{\bcw{Z}}$ are  sheaves of ideals of $Z$ and $\bcw Z$. 
\end{lem}

\begin{proof}
Let $\cla{M}=\cla{O}_X/\cla{I}_Z$ and $\bcw{\cla{M}}=\bcpb{u}\cla{M}=\cla{O}_{\bcw{X}}/\cla{I}_{\bc{u}^{-1}(Z)}$.
Then $\bcw{\cla{M}}^{pt}=\bcw{\cla{M}}/\bcw{\cla{N}}=\cla{O}_{\bcw{X}}/\cla{I}_{\bcw{Z}'}$
for a sheaf of ideals $\cla{I}_{\bcw{Z}'}$
that determines a closed subscheme $\bcw{Z}'\subseteq \bc{u}^{-1}(Z)$. 
We claim that $\bcw{Z}'=\bcw{Z}$.
Since  $\bcw{\cla{N}}$ vanishes on $\bc{f}^{-1}(\bcw{U})$, 
$\cla{I}_{\bcw{Z}'}|_{\bc{f}^{-1}(\bcw{U})}=\cla{I}_{\bc{u}^{-1}(Z)}|_{\bc{f}^{-1}(\bcw{U})}$
and
$$
\bcw{Z}'\cap \bc{f}^{-1}(\bcw{U}) = \bc{u}^{-1}(Z) \cap \bc{f}^{-1}(\bcw{U})=\bc{u}^{-1}(Z\cap f^{-1}(U)).
$$
From the definition of the pure transform, $\bcw{Z}\subseteq \bcw{Z}'$. Furthermore, $\bcw{Z}'\cap \bc{f}^{-1}(\bcw{U})=\bcw{Z}\cap \bc{f}^{-1}(\bcw{U})$, and hence
$
\brk{\cla{O}_{\bcw{X}}/\cla{I}_{\bcw{Z}'}}|_{\bc{f}^{-1}(\bcw{U})}=\brk{\cla{O}_{\bcw{X}}/\cla{I}_{\bcw{Z}}}|_{\bc{f}^{-1}(\bcw{U})}.
$

Assume that  $\bcw{Z}\ne \bcw{Z}'$.
Then we can find
an affine open set $V=\spec A$ in $\bcw{X}$ 
such that
$\bcw{Z}'\cap V=V(\cl{a})$, $\bcw{Z}\cap V=V(\cl{b})$ with $\cl{a}\subset \cl{b}$ ideals of $A$ (strict inclusion). Let  $a\in \cl{b}\setminus \cl{a}$, so that $a$ is zero in $A/\cl{b}$, but non-zero in $A/\cl{a}$. Let $\cl{p}$ be in $f^{-1}(\bcw{U}) \cap V$. If $\cl{p}\not \in \bcw{Z}'$ then clearly $\brk{\cla{O}_{\bcw{X}}/\cla{I}_{\bcw{Z}'}}_\cl{p}=0$. If $\cl{p}$ is in 
$\bcw{Z}'\cap\bc{f}^{-1}(\bcw{U})\cap V=\bcw{Z}\cap\bc{f}^{-1}(\bcw{U})\cap V$, then since $\brk{\cla{O}_{\bcw{X}}/\cla{I}_{\bcw{Z}'}}|_{\bc{f}^{-1}(\bcw{U})}=\brk{\cla{O}_{\bcw{X}}/\cla{I}_{\bcw{Z}}}|_{\bc{f}^{-1}(\bcw{U})}$,
$$
a=0\text{ in }\brk{\cla{O}_{\bcw{X}}/\cla{I}_{\bcw{Z}'}}_\cl{p}= \brk{\cla{O}_{\bcw{X}}/\cla{I}_{\bcw{Z}}}_\cl{p}= (A/\cl{b})_{\cl{p}}.
$$
Thus $a$ is a non-zero local section supported outside $\bc{f}^{-1}(\bcw{U})$. This contradicts the definition of the pure transform of $\cla{M}$, hence $\bcw{Z}'=\bcw{Z}$.
\end{proof}

\begin{prop} \label{puretransformequiv}
$\cla{M}^{pt}$ is a $G$-sheaf on $\bcw{X}$ and $\bcw{Z}$ is a $G$-stable closed subscheme.
\end{prop}

\begin{proof}
$\bcw{X}$ has a natural structure of a $G$-scheme and $\bcw{\cla{M}}$ of a $G$-sheaf
such that $\bc{f}$ and $\bc{u}$ are $G$-morphisms.
It suffices to show that $\bcw{\cla{N}}\subseteq \bcw{\cla{M}}$ is a $G$-subsheaf. Write
$$
\bc{\mu}:G\times_S \bcw{X}\rightarrow \bcw{X}\quad \text{and}\quad \bcw{\pr}_2:G\times_S \bcw{X}\rightarrow \bcw{X}
$$
for the multiplication map and the second projection of $G\times_S \bcw{X}$, respectively, and
$$
\alpha:\bcpb{\mu}\bcw{\cla{M}}\rightarrow \bcwpb{\pr}_2\bcw{\cla{M}}
$$
for the isomorphism of $\cla{O}_{G\times_S\bcw{X}}$-modules that defines the $G$-structure on $\bcw{\cla{M}}$. We want to show that $\alpha(\bcpb{\mu}\bcw{\cla{N}})\subseteq \bcwpb{\pr}_2\bcw{\cla{N}}$. Since $\bc{\mu}$ is flat,
$\bcpb{\mu}\bcw{\cla{N}}$ is the subsheaf of sections of $\bcpb{\mu}\bcw{\cla{M}}$ supported outside $\bc{\mu}^{-1}(\bc{f}^{-1}(\bcw{U}))$, and similarly for $\bcw{\pr}_2^{\ast}\cla{N}$ \cite[II, Ex. 1.20]{H}. Note that
$$
\bc{\mu}^{-1}(\bc{f}^{-1}(\bcw{U}))=G\times_S \bc{f}^{-1}(\bcw{U})=\bcw{\pr}_2^{-1}(\bc{f}^{-1}(\bcw{U})),
$$
as $\bc{f}^{-1}(\bcw{U})$ is $G$-stable. 
Thus the isomorphism $\alpha$ preserves the set of sections supported outside of $\bc{\mu}^{-1}(\bc{f}^{-1}(\bcw{U}))$,
therefore $\alpha(\bcpb{\mu}\bcw{\cla{N}})\subseteq \bcwpb{\pr}_2\bcw{\cla{N}}$ as required. 

For closed subschemes, 
the quotient $\cla{O}_X/\cla{I}_Z$ is a $G$-sheaf, and  so by the above is its pure transform $\cla{O}_{\bcw{X}}/\cla{I}_{\bcw{Z}}$. Thus  $\bcw{Z}$ is a $G$-stable closed subscheme of $\bcw{X}$.
\end{proof}

\begin{lem} \label{propertranspreimage}\label{stricttransformflat}
$\bcw{\cla{M}}^{pt}=
\bcw{\cla{M}}$ 
(resp.~$\tilde Z=\bc{u}^{-1}(Z)$) 
if  $\cla{M}$ is flat over $Y$ 
(resp.~$f|_Z$ is flat)
and $\ass(\bcw{Y})\subseteq \bcw{U}$, which holds for example 
if $\bcw{Y}$ is integral.
\end{lem}

\begin{proof}
If  $\bcw{\cla{M}}$ is flat, associated points of $\bcw{\cla{M}}$ map to associated points of~$\bcw{Y}$ \cite[Th.~3.3.1]{EGAIV}, hence $\ass(\bcw{\cla{M}})\subseteq \bc{f}^{-1}(\bcw{U})$.
Flatness of $f|_Z$ is equivalent to flatness of $\cla{O}_X/\cla{I}_{Z}$. The pure transform of $\cla{O}_X/\cla{I}_{Z}$ with respect to $U$ is then $\bcpb{u}(\cla{O}_X/\cla{I}_{Z})=\cla{O}_{\bcw{X}}/\cla{I}_{\bc{u}^{-1}(Z)}$, 
and at the same time $\cla{O}_{\bcw{X}}/\cla{I}_{\bcw{Z}}$. It follows that $\bcw{Z}=\bc{u}^{-1}(Z)$.
\end{proof}

\begin{defn}
A projective $G$-morphism $u:\,\bcw{Y}\to Y$  is called
an \emph{equivariant flattening} of~$\cla{M}$ 
(resp.~of $Z$) if $\bcw{\cla{M}}^{pt}$ is flat over $\bcw{Y}$
(resp.~$\bc{f}|_{\bcw{Z}}$ is flat).
\end{defn}

\begin{conj}
An equivariant flattening always exists.
\end{conj}

We prove existence of an equivariant flattening under an additional assumption that
$f$ is a {\em projective} morphism, which we will assume from now on.
Let $\cl{Quot}_{\cla{M}/X/Y}$ be the Quot functor, i.e. the contravariant functor $\mathsf{Sch}_Y\rightarrow \mathsf{Set}$ such that
$$
\cl{Quot}_{\cla{M}/X/Y}(T)=\crb{\begin{array}{c} \text{Coherent quotients of the pullback of} \\ \cla{M}\text{ to }T\times_Y X \text{ that are flat over }T \end{array}}.
$$
It is represented by the Quot scheme $\mathrm{Q}$,
a disjoint union $\coprod_i \mathrm{Q}_i$ of projective schemes $\mathrm{Q}_i$ over $Y$ (see \cite{TDTE}).
Write $\pi:\mathrm{Q}\rightarrow Y$ for the structure morphism. 
$\mathrm{Q}$ can be viewed as a scheme over $S$ via the composition of $\pi$ with $Y\rightarrow S$.

\begin{lem} \label{quotgstruct}
$\mathrm{Q}$ has a natural structure of a $G$-scheme and 
$\pi$ is a $G$-morphism.
\end{lem}
\begin{proof}
Given a scheme $T$, we define an action of $G_S(T)$ on $\mathrm{Q}_S(T)$, functorial in~$T$, as follows. Let $g\in G_S(T)$ and $s\in \mathrm{Q}_S(T)$. 
We view $T$ as a scheme over $Y$ via $y=\pi \circ s$, in which case $s$ is a morphism over $Y$, and $y$ a morphism over $S$:
$$
\begindc{\commdiag}[45]
\obj(8,0)[s]{$S$} \obj(8,7)[y]{$Y$} \obj(16,12)[q]{$\mathrm{Q}$} \obj(0,12)[t]{$T$}
\mor{t}{q}{$s$} \mor{t}{y}{$y$} \mor{q}{y}{$\pi$}[-1,0] \mor{y}{s}{} \mor{t}{s}{} \mor{q}{s}{}
\enddc
$$
The $Y$-morphism $s$ corresponds to a coherent quotient $\cla{N}$ of $\bcpb{y}\cla{M}$ flat over $T$:
\begin{equation}
\begindc{\commdiag}[40]
\obj(0,8)[xp]{$T\times_Y X$} \obj(0,0)[yp]{$T$} \obj(16,0)[y]{$Y$} \obj(16,8)[x]{$X$}
\mor{xp}{yp}{}[-1,0] \mor{xp}{x}{$\bc{y}$} \mor{yp}{y}{$y$} \mor{x}{y}{$f$}
\enddc\label{rgrgr}
\end{equation}

The morphism $T\times_Y X\rightarrow T$ induces a map $G_S(T)\rightarrow G_S(T\times_Y X)$. Let $\bc{g}$ be the image of $g$ under this map. Note that $\bc{y}\in X_S(T\times_Y X)$, so that $\bc{g}\bc{y}=\bcw{gy}$ is also an element in $X_S(T\times_Y X)$, where $\bcw{gy}$ is given by the cartesian diagram
$$
\begindc{\commdiag}[40]
\obj(0,8)[xp]{$T\times_Y X$} \obj(0,0)[yp]{$T$} \obj(16,0)[y]{$Y$} \obj(16,8)[x]{$X$}
\mor{xp}{yp}{}[-1,0] \mor{xp}{x}{$\bcw{gy}$} \mor{yp}{y}{$gy$} \mor{x}{y}{$f$}
\enddc
$$
(here $T\times_Y X$ and $T\times_Y X\rightarrow T$ are as in \eqref{rgrgr}).
Since $\cla{M}$ is a $G$-sheaf, there is an isomorphism of sheaves on $T\times_Y X$:
$$
\phi: \bcpb{y}\cla{M}\rightarrow \bcwpb{gy}\cla{M}
$$

The quotient sheaf $\cla{N}$ is identified via $\phi$ with a coherent quotient sheaf of $\bcwpb{gy}\cla{M}$ that is flat over $T$. This gives a point in $\mathrm{Q}_Y(T)\subseteq \mathrm{Q}_S(T)$, where $T$ is a scheme over $Y$ via $gy$. We define $gs$ to be this point. Showing the properties of a group action and functoriality on $T$ is easy and is omitted.
Finally, we show that $\pi$ is a $G$-morphism. Let $T$ be a scheme. Let $\pi_T:\mathrm{Q}_S(T)\rightarrow Y_S(T)$ be the map induced by $\pi$ on $T$-points, and let $g\in G_S(T)$, $s\in \mathrm{Q}_S(T)$. Let $y$ be the image of $s$ in $Y_S(T)$. From the definition of $gs$, $\pi_T(gs)=gy=g\pi_T(s)$, and so $\pi$ is a $G$-morphism.
\end{proof}

\begin{lem} \label{equivpbqpt}
Let $R$ be a $G$-scheme and $y:R\rightarrow Y$ a $G$-morphism 
such that the coherent sheaf $\bcpb{y}\cla{M}$ is flat over $R$:
$$
\begindc{\commdiag}[40]
\obj(0,8)[xp]{$X\times_Y R$} \obj(0,0)[yp]{$R$} \obj(16,0)[y]{$Y$} \obj(16,8)[x]{$X$}
\mor{xp}{yp}{}[-1,0] \mor{xp}{x}{$\bc{y}$} \mor{yp}{y}{$y$} \mor{x}{y}{$f$}
\enddc
$$
Then the corresponding 
morphism $s:R\rightarrow \mathrm{Q}$ over $Y$
is  a $G$-morphism.
\end{lem}

\begin{proof}
For a scheme $T$, write $s_T:R_S(T)\rightarrow \mathrm{Q}_S(T)$ for the induced map on $T$-points. 
Given $g\in G_S(T)$ and $r\in R_S(T)$, we want to show that $s_T(gr)=gs_T(r)$. The image $s_T(r)$ is the point in $\mathrm{Q}_S(T)=\cl{Quot}_{\cla{M}/X/Y}(T)$  associated to the sheaf $\bcpb{r}\bcpb{y}\cla{M}=(\bc{y}\circ \bc{r})^{\ast}\cla{M}$ on $X\times_Y T$ 
flat over $T$:
$$
\begindc{\commdiag}[40]
\obj(-16,8)[xr]{$T\times_Y X$} \obj(-16,0)[yr]{$T$}\obj(0,8)[xp]{$X\times_Y R$} \obj(0,0)[yp]{$R$} \obj(16,0)[y]{$Y$} \obj(16,8)[x]{$X$}
\mor{xp}{yp}{}[-1,0] \mor{xp}{x}{$\bc{y}$} \mor{yp}{y}{$y$} \mor{x}{y}{$f$}
\mor{yr}{yp}{$r$} \mor{xr}{xp}{$\bc{r}$} \mor{xr}{yr}{}
\enddc
$$

Let $\bc{g}\in G_S(T\times_Y X)$ be the image of $g$ under the map $G_S(T)\rightarrow G_S(T\times_Y X)$ induced by $T\times_Y X\rightarrow T$. Note that $\bc{y}\circ \bc{r}\in X_S(T\times_Y X)$ and, as in the proof of Lemma \ref{quotgstruct}, $\bc{g}(\bc{y}\circ \bc{r})=\brk{g(y\circ r)}^{\sim}=\brk{y\circ gr}^{\sim}$:
$$
\begindc{\commdiag}[40]
\obj(0,8)[xp]{$T\times_Y X$} \obj(0,0)[yp]{$T$} \obj(16,0)[y]{$Y$} \obj(16,8)[x]{$X$}
\mor{xp}{yp}{}[-1,0] \mor{xp}{x}{$\bcw{y\circ gr}$} \mor{yp}{y}{$y\circ gr$} \mor{x}{y}{$f$}
\enddc
$$
The equality $g(y\circ r)=y\circ gr$ follows from the equivariance of $y$. There is an isomorphism of sheaves on $T\times_Y X$:
$$
\phi: \brk{\bc{y}\circ \bc{r}}^{\ast}\cla{M}\rightarrow \bcwpb{y\circ gr}\cla{M}
$$
This is a coherent sheaf, flat over $T$, that determines the point $gs_T(r)$ in $\mathrm{Q}_S(T)$.

The image $s_T(gr)$ is the point in $\mathrm{Q}_S(T)$ associated to the sheaf $((y\circ gr)^{\sim})^{\ast}\cla{M}$ on $X\times_Y T$.
This is precisely $gs_T(r)$, thus $s$ is a $G$-morphism.
\end{proof}

\begin{thm} \label{equivflat}\label{eqclosedflat} \label{rgflat}\label{closedflat}
An equivariant flattening of $\cla{M}$ (resp.~of $Z$) exists for any 
projective $G$-morphism~$f:\,X\to Y$. If~$Y$ is integral then we can assume that 
$\bcw{Y}$ integral.
\end{thm}

\begin{proof}
When $G=S$ is a trivial group scheme, see \cite[Chap. 4, \S1, Thm. 1]{R}.
In general, we follow the same strategy.
By Lemma \ref{equivpbqpt}, the sheaf $\cla{M}|_{f^{-1}(U)}$ induces a $G$-morphism $v:U\rightarrow \mathrm{Q}$ over~$Y$. 
Let $\bcw{Y}$ be the scheme-theoretic image of $v$, a $G$-stable closed subscheme of $\mathrm{Q}$, and let $w:U\rightarrow \bcw{Y}$ be the induced $G$-morphism, and $s:\bcw{Y}\emb \mathrm{Q}$ the associated closed $G$-embedding. Write $u:\bcw{Y}\rightarrow Y$ for the structure morphism.
$$
\begindc{\commdiag}[50]
\obj(0,12)[u]{$U$} \obj(8,7)[y]{$\bcw{Y}$} \obj(16,12)[x]{$\mathrm{Q}$} \obj(8,0)[s]{$Y$} 
\obj(8,1)[ghost1]{} \obj(8,8)[ghost2]{}
\mor{u}{y}{$w$} \mor{u}{x}{$v$} \mor{y}{x}{$s$}[1,5] \mor{u}{s}{}[-1,5] \mor{y}{s}{} \mor{x}{s}{$\pi$} \mor{ghost1}{ghost2}{$u$}[0,-1]
\enddc
$$

We claim that $u$ is an equivariant flattening of $\cla{M}$. Since $Y$ is noetherian, $U$ has finitely many irreducible components, and the same holds for its image $\bcw{Y}$. Therefore $\bcw{Y}$ lies in finitely many irreducible components of $\mathrm{Q}$, so that $u$ is projective. If in addition $Y$ is integral, there is only one irreducible component, and so $\bcw{Y}$ is integral. Furthermore, $u=\pi\circ s$ is a $G$-morphism.
The composition $u\circ w$ is the open $G$-embedding $U\emb Y$, hence $w$ is also an open $G$-embedding. Let $\bcw{U}=w(U)$, which is a $G$-stable open set in $\bcw{Y}$. 
In summary, $u$ is a projective birational $G$-morphism, and it restricts to an isomorphism on $G$-stable open sets $\bcw{U}\isom U$.
The morphism $s:\bcw{Y}\emb \mathrm{Q}$ 
corresponds to a quotient sheaf $\cla{P}=\bcw{\cla{M}}/\bcw{\cla{N}}$ on $\bcw{X}=X\times_Y \bcw{Y}$, where 
$\bcpb{u}$ is defined by diagram \eqref{sefvwerv}.
We will show that $\cla{P}$ is the pure transform of $\cla{M}$.

The morphism $w:U\emb \bcw{Y}$ 
induces a map $\mathrm{Q}_Y(\bcw{Y})\rightarrow \mathrm{Q}_Y(U)$
that 
sends a coherent quotient of the pullback of $\cla{M}$ on $X\times_Y \bcw{Y}$ that is flat over $\bcw{Y}$ to its pullback on $f^{-1}(U)$, which is  a coherent quotient of $\cla{M}|_{f^{-1}(U)}$ that is flat over $U$.
Thus the image of $s$ in $\mathrm{Q}_Y(U)$, which is $w\circ s=v$, corresponds to $\bcpb{w}\cla{P}=\cla{M}|_{f^{-1}(U)}$.
As an open immersion, $w$ is flat, hence
$$
\bcpb{w}\cla{P}=\bcpb{w}\bcw{\cla{M}}/\bcpb{w}\bcw{\cla{N}} = (\bcpb{w}\bcpb{u}\cla{M})/\bcpb{w}\bcw{\cla{N}} = 
\cla{M}|_{f^{-1}(U)}/ \bcpb{w}\bcw{\cla{N}}.
$$
We deduce that $\bcpb{w}\bcw{\cla{N}}$ is the zero sheaf. Taking the pullback of $\bcpb{w}\bcw{\cla{N}}$ by the isomorphism $\bc{u}|_{\bc{f}^{-1}(\bcw{U})}:\bc{f}^{-1}(\bcw{U})\isom f^{-1}(U)$, we see that 
$\bcw{\cla{N}}$ vanishes on $\bc{f}^{-1}(\bcw{U})$.
Since $\bcw{Y}$ is the scheme-theoretic image of $U$, its associated points are contained in $\bcw{U}$.
Due to the flatness of $\cla{P}$, the associated points of $\cla{P}$ map to associated points of $\bcw{Y}$ \cite[Th.~3.3.1]{EGAIV}. We deduce $\ass(\cla{P})\subseteq \bc{f}^{-1}(\bcw{U})$, and Definition \ref{stricttransformchar} implies that $\cla{P}$, which is flat over $\bcw{Y}$, is the pure transform of $\cla{M}$.
For a $G$-stable closed subscheme $Z\subset X$,
the proof is the same using Lemma~\ref{puretransagree}.
\end{proof}

To prove existence of tropical compactifications in the next subsection, we 
need to prove existence of equivariant flattening of the multiplication map.

\begin{cor}\label{mainthmoverk}
Let $X$ be an integral $G$-scheme of finite type
and 
let $\cla{M}$ be a coherent sheaf on $X$.
Consider a coherent sheaf 
$\cla{N}=\pr_2^{\ast}\cla{M}$ on $G\times_S X$.
Suppose that
$G$ admits 
an equivariant compactification $G\emb G'$  in a projective scheme $G'$\footnote{
The equivariant compactification $G'$ exists if 
$G$ is a surjective smooth affine group scheme with connected fibers over a normal noetherian scheme $S$
\cite[Thm.~4.9]{Su} or if $G$ is a linear algebraic group over 
$S=\spec k$ for an algebraically closed field $k$  \cite[Thm.~3]{Suv}.}.
Then there exists a projective birational $G$-morphism
$u:\bcw{X}\rightarrow X$ 
from an integral scheme $\bcw{X}$ such that the $G$-sheaf 
$\bcw{\cla{N}}^{pt}$  on $G\times_S \bcw{X}$ (defined in the proof) is flat over $\bcw{X}$ via $\bcw{\mu}$:
\begin{equation}\label{dfbDfhdh}
\begindc{\commdiag}[40]
\obj(0,8)[gu]{$G\times_S \bcw{X}$} \obj(0,0)[gx]{$\bcw{X}$} \obj(16,0)[x]{$X$} \obj(16,8)[u]{$G\times_S X$}
\mor{gu}{gx}{$\bc{\mu}$}[-1,0] \mor{gu}{u}{$\bc{u}$} \mor{gx}{x}{$u$} \mor{u}{x}{$\mu$}
\enddc
\end{equation}
\end{cor}

\begin{proof}
Let $G$ act on $G\times_S X$ by multiplication on the first factor, $g(h,x)=(gh,x)$, so that 
the multiplication map $\mu$  is a $G$-morphism.
Then $\cla{N}$
has a canonical  $G$-sheaf structure.
Let $G\times'_S X\simeq G\times_S X$ as a scheme but $G$ acts by multiplication on both factors, $g(h,x)=(gh,gx)$.
Consider the $G$-isomorphism $\phi$ given by
$$
\phi:G\times'_S X\isom G\times_S X,\quad (g,x)\mapsto (g,g^{-1}x).
$$
The scheme $G\times'_S X$ is a $G$-stable open subset of $G'\times'_S X$. There exists 
a coherent $G$-sheaf $\cla{P}$ on $G'\times'_S X$ such that $\cla{P}|_{G\times'_S X}=\phi^{\ast}\cla{N}$. Indeed, 
the pushforward of $\phi^{\ast}\cla{N}$ to $G'\times_S X$ is a quasi-coherent $G$-sheaf, which
is a
direct limit of its coherent $G$-subsheaves that restrict to $\phi^{\ast}\cla{N}$ on $G\times_S X$ (see \cite[Cor. 2.4]{Th}
or \cite[Cor. 15.5]{LM}).
The second projection $\pr_2:G'\times_S X\rightarrow X$ is projective as a base change of $G'\rightarrow S$. 
Since $X$ is integral, there exists a a $G$-invariant dense open subset
$U\subset X$ such that $\cla{P}|_{G'\times'_SU}$ is flat over $U$ via $\pr_2$.

A~morphism $u:\bcw{X}\rightarrow X$ is a $G$-flattening of $\cla{N}$ with respect to $(\mu,U)$ if and only if it is a $G$-flattening of $\phi^{\ast}\cla{N}$ with respect to $\pr_2=\mu\circ \phi$.  
By Theorem \ref{equivflat}, there is a $G$-flattening  $u:\bcw{X}\rightarrow X$ of $\cla{P}$ with integral $\bcw{X}$:
$$
\begindc{\commdiag}[40]
\obj(0,8)[gu]{$G'\times'_S \bcw{X}$} \obj(0,0)[gx]{$\bcw{X}$} \obj(16,0)[x]{$X$} \obj(16,8)[u]{$G'\times'_S X$}
\mor{gu}{gx}{$\bcw{\pr}_2$}[-1,0] \mor{gu}{u}{$\bc{u}$} \mor{gx}{x}{$u$} \mor{u}{x}{$\pr_2$}
\enddc
$$
Restriction to the $G$-stable open set $G\times'_S X$ shows that $u$ is a $G$-flattening of $\cla{P}|_{G\times'_S X}=\phi^{\ast}\cla{N}$ with respect to $\pr_2$, hence a $G$-flattening of $\cla{N}$ with respect to~$\mu$.
\end{proof}

\subsection{General Tropical Compactifications} \label{secttropcompactif}\label{sectflatmultiplmap}

In this subsection the base scheme $S$ is $\spec k$ 
for an algebraically closed field $k$, $G$  is a smooth linear
algebraic group over $k$ and $U$ is a homogeneous variety of $G$.
We fix a closed subscheme $Y\subseteq U$
and consider open dense $G$-embeddings $U\emb X$.

\begin{lem}\label{aefvweve}
The multiplication map 
$\mu_Y: G\times Y\rightarrow U$ is faithfully flat. 
\end{lem}

\begin{proof}
Consider the morphism
$$
\psi:\,G\times U\rightarrow U\times U,\quad (g,u)\mapsto (gu,u).
$$
Since $\psi$ is equidimensional and $G\times U$, $U\times U$ are smooth,
$\psi$ is flat \cite[6.1.5]{EGAIV}.
Thus its base change $G\times Y\rightarrow U\times Y$ is flat, and therefore
its composition with the projection $U\times Y\to U$, which is 
$\mu_Y$, is flat. It is clearly surjective. 
\end{proof}

\begin{defn}\renewcommand{\labelenumi}{(\roman{enumi})}
Let $X$ be a $G$-variety containing $U$ as a dense open subset.
The scheme-theoretic closure $\overline{Y}\subseteq X$ is called a \emph{tropical compactification} of $Y$
 if $\overline{Y}$ is proper and the multiplication map 
$ \mu_{\overline{Y}}:G\times \overline{Y}\rightarrow X$
 is faithfully flat.
\end{defn}

\begin{prop} \label{puretransformmultagree}
Consider a projective birational $G$-morphism $u:X\rightarrow X'$.
Let $\overline{Y}'$ be the scheme-theoretic closure of $Y$ in $X'$. Consider a cartesian diagram
$$
\begindc{\commdiag}[40]
\obj(0,8)[gu]{$G\times X$} \obj(0,0)[gx]{$X$} \obj(16,0)[x]{$X'$} \obj(16,8)[u]{$G\times X'$}
\mor{gu}{gx}{$\mu$}[-1,0] \mor{gu}{u}{$\bc{u}$} \mor{gx}{x}{$u$} \mor{u}{x}{$\mu'$}
\enddc
$$
The pure transform of $G\times\overline{Y}'$ with respect to $(u,U)$ is $G\times\overline{Y}$.
\end{prop}

\begin{proof}
By Lemma~\ref{aefvweve}, the pure transform of $G\times\overline{Y}'$ is 
the closure of subscheme
$$
\bc{u}^{-1}\brk{(G\times\overline{Y}') \cap (G\times U)}= 
\bc{u}^{-1}\brk{G\times Y)}=G\times u^{-1}(Y)
$$
in $G\times X$, which is  $G\times \overline{Y}$.
\end{proof}

Given one tropical compactification, we can get more using the 
following:
\begin{prop}\label{tropcompbir}
Let $\overline{Y}\subseteq X$ be a tropical compactification of $Y$ and $u:\,\bcw{X}\rightarrow X$ a proper birational $G$-morphism
of varieties. 
Then the scheme-theoretic closure $\bcw{Y}$ of $Y$ in $\bcw{X}$ is a tropical compactification of $Y$
and is equal to $u^{-1}(\overline{Y})$.
\end{prop}

\begin{proof}
Let $\mu:G\times X\rightarrow X$ be the multiplication map of $X$. Consider the coherent sheaves $\cla{M}=\cla{O}_X/\cla{I}_{\overline{Y}}$ on $X$ and $\cla{N}=\pr_2^{\ast}\cla{M}=\cla{O}_{G\times X}/\cla{I}_{G\times\overline{Y}}$ on $G\times X$. 
Faithful flatness of $\mu_{\overline{Y}}$ is equivalent to faithful flatness of $\cla{N}$ with respect to $\mu$. 
The pure transform of $\cla{N}$ is $\bc{u}^{\ast}\cla{N}=\cla{O}_{G\times \bcw{X}}/\cla{I}_{G\times \bcw{Y}}$ (Lemmas \ref{stricttransformflat}, \ref{puretransagree} and Prop.~\ref{puretransformmultagree}), which is faithfully flat. This in turn is equivalent to faithful flatness of $\bc{\mu}:G\times \bcw{Y}\rightarrow \bcw{X}$. Furthermore, $\bcw{Y}=u^{-1}(\overline{Y})$ (Lemma \ref{propertranspreimage}), and so $\bcw{Y}$ is proper since $u$ is. Thus $\bcw{Y}\subseteq \bcw{X}$ is a tropical compactification.
\end{proof}

Using Prop.~\ref{tropcompbir}, we can partially order tropical compactifications under the relation
$\bcw{Y}\succeq \overline{Y}$ 
if there is a proper birational $G$-morphism $\bcw{X}\rightarrow X$.
This is an analog of \cite[Prop. 2.5]{Te} in the toric case, where 
this ordering has a combinatorial meaning:  refinement of the fan of the toric variety $X$. 
In the next section we will see that the same holds for tropical compactifications in spherical varieties.

\begin{thm}\label{tropcompexist}\label{tropcompexistvar}
There exists a tropical compactification $\overline{Y}$ in a normal $G$-variety~$X$.
\end{thm}

\begin{proof}
By \cite[Thm. 3]{Suv}, we can find an equivariant compactification
$U\emb X'$  in a $G$-variety $X'$.
Let $\overline{Y}'\subseteq X'$ be the scheme-theoretic closure of $Y$. 
We claim that we can find a projective birational $G$-morphism $u:\,X\rightarrow X'$ 
such that the multiplication map $\bc{\mu}_{\overline{Y}}:G\times \overline{Y} \rightarrow X$ is flat, where $\overline{Y}$ is the scheme-theoretic closure of $Y$ in~$X$.
Indeed, 
consider the coherent $G$-sheaf $\cla{O}_{G\times X'}/\cla{I}_{G\times \overline{Y}'}=\pr_2^{\ast}(\cla{O}_{X'}/\cla{I}_{\overline{Y}'})$. 
Flatness of $\mu_Y$ over $U$ (Lemma~\ref{aefvweve}) 
is equivalent to flatness of $(\cla{O}_{G\times X'}/\cla{I}_{G\times \overline{Y}'})|_{G\times U}$ via $\mu$.
By \cite[Thm. 3]{Suv}, we can find an equivariant compactification
$G\emb G'$.
Apply Corollary \ref{mainthmoverk} to get an equivariant flattening $u:\,X\rightarrow X'$ of 
$\cla{O}_{G\times X'}/\cla{I}_{G\times\overline{Y}'}$:
$$
\begindc{\commdiag}[40]
\obj(0,8)[gu]{$G\times X$} \obj(0,0)[gx]{$X$} \obj(16,0)[x]{$X'$} \obj(16,8)[u]{$G\times X'$}
\mor{gu}{gx}{$\mu$}[-1,0] \mor{gu}{u}{$\bc{u}$} \mor{gx}{x}{$u$} \mor{u}{x}{$\mu'$}
\enddc
$$
Flatness of the pure transform of $\cla{O}_{G\times X'}/\cla{I}_{G\times\overline{Y}'}$, 
which is $\cla{O}_{G\times X}/\cla{I}_{G\times \overline{Y}}$ by Lemma \ref{puretransagree} and Proposition \ref{puretransformmultagree}, is equivalent to flatness of $\mu_{\overline{Y}}:G\times \overline{Y}\rightarrow X$.

Using Proposition~\ref{tropcompbir}, we can assume that $X$ is normal.
Since $U\subset X'$ is an open $G$-orbit,
$u$ restricts to an isomorphism on $G$-stable open dense sets $U'\isom U$
and so we can view $X$ as an equivariant compactification of $U$.
Since $u$ is projective and $X'$ is proper, $X$ is also proper, and so is $\overline{Y}$. 

Finally, since $\mu_{\overline{Y}}$ is flat, its image is open in $X$. 
To make $\mu_{\overline{Y}}$ faithfully flat (i.e.~flat and surjective),
it remains to substitute $X$ with the image of $\mu_{\overline{Y}}$.
\end{proof}

\section{Compactifications of Subvarieties of Spherical Homogenous Spaces} \label{sphericaltrop}

\subsection{Brief Survey of Spherical Varieties} \label{spherreview}

Let $\cla{X}=\hom(B,k^{\times})$ be the group of characters. Let $k(G/H)^{(B)}$ be the multiplicative group of $B$-semi-invariant rational functions on $G/H$. It
comes with a homomorphism
$
\Xi:\,k(G/H)^{(B)}\rightarrow \cla{X}$,
$f\mapsto \chi_f$:
$$
k(G/H)^{(B)}=\crb{f\in k(G/H)^{\times} \text{ such that  }gf=\chi_f(g)f\text{ for all }\,g\in B}.
$$
The kernel of $\Xi$ is the set of constant functions, hence $k(G/H)^{(B)}/k^{\times}$ injects in~$\cla{X}$. 
The image $\Lambda\subset\cla{X}$ is called the weight lattice.
Its rank is called the rank of $G/H$.
Any $\qq$-valuation of $k(G/H)$ trivial on $k$ can be restricted to $k(G/H)^{(B)}$ and 
induces a homomorphism $\Lambda\rightarrow \qq$, i.e. an element of $\mathrm{Q}=\hom(\Lambda,\qq)$.
This gives a map
$$
\varrho:\crb{\qq\text{-valuations of }k(G/H)}\rightarrow \mathrm{Q}.
$$
Restricting to geometric valuations, we can 
view $\varrho$ as a map from the set of prime divisors $D$ of $G/H$ 
(or its birational model), 
sending $D$ to $\varrho(v_D)$, where $v_D$ is the corresponding valuation.
Denote by $\cls{V}$ the set of $G$-invariant valuations of $k(G/H)$. 
These valuations are automatically geometric.  
Then $\varrho$ restricts to an injection on~$\cls{V}$ and
we identify $\cls{V}$ with its image in $\mathrm{Q}$. This 
is a rational convex polyhedral cone,
called the \emph{valuation cone}. 
Let $\cls{D}$ be the finite set of $B$-stable prime divisors of $G/H$. 
The elements of $\cls{D}$ are called \emph{colors}.
\begin{defn}
A 
colored cone is a pair $(\cls{C},\cls{F})$, where $\cls{F}\subseteq \cls{D}$ is a subset of colors and $\cls{C}\subseteq \mathrm{Q}$ 
is a strictly convex cone 
generated by $\varrho(\cls{F})$ and finitely many elements of~$\cls{V}$. Furthermore, 
$\cls{C}^{\circ}\cap\cls{V}\ne\emptyset$ and
$0\not\in\varrho(\cls{F})$, where we denote by $\cls{C}^{\circ}$ the relative interior.

A face of a colored cone $(\cls{C},\cls{F})$ is a colored cone $(\cls{C}_0,\cls{F}_0)$ such that $\cls{C}_0$ is a face of $\cls{C}$ that intersects $\cls{V}$ non-trivially and $\cls{F}_0=\cls{F}\cap \varrho^{-1}(\cls{C}_0)$.

A colored fan $\cla{F}$ is a non-empty set of colored cones 
such that
every face of a cone in $\cla{F}$ is in $\cla{F}$ and any element $v\in\cls{V}$ lies in the interior of at most one cone.
\end{defn}

A spherical variety is called \emph{simple} if it contains a unique closed $G$-orbit. Any spherical variety is covered 
by finitely many simple spherical open subvarieties.
There is a bijection
between spherical embeddings $G/H\hookrightarrow X$ and colored fans 
that restricts to a bijection
between simple spherical embeddings and colored cones.
Namely, let $X$ be a simple spherical variety with a unique closed $G$-orbit $Y$. 
Let $\cls{B}_X\subseteq \cls{V}$ be the set of $G$-stable prime divisors of $X$ containing $Y$, and let $\cls{F}_X$ 
be the set of $B$-stable prime divisors of $X$ containing $Y$ that are not $G$-stable. 
We identify any $D\in \cls{F}_X$ with the intersection $D\cap (G/H)$, which is a non-empty $B$-stable prime divisor of $G/H$, i.e.~a color. 
Let $\cls{C}_X\subset \mathrm{Q}$ be the cone generated by $\cls{B}_X$ and $\varrho(\cls{F})_X$. 
Then $(\cls{C}_X,\cls{F}_X)$ is the colored cone associated to $X$.

\begin{defn} \label{toroiddef}
A spherical variety $X$ is called toroidal if  the associated colored fan has no colors, 
i.e. if $\cls{F}=\varnothing$ for any colored cone $(\cls{C},\cls{F})$ of the colored fan. 
\end{defn}

For any spherical variety $X$, there is a surjective proper birational $G$-morphism $X'\surj X$, that restricts to the identity on $G/H$, with $X'$ a toroidal spherical variety. 

The \emph{support} of a colored fan $\cla{F}$ is the intersection of the union of its cones with~$\cls{V}$.
A spherical variety is complete if and only if the support of its colored fan is ~$\cls{V}$.
An \emph{equivariant compactification} of a spherical variety $X$ is a complete spherical variety $X'$ (for the same  space $G/H$) with an open dense $G$-embedding $X\emb X'$. 
As for toric varieties, one can  find an equivariant compactification $X'$ by completing the colored fan of~$X$.
If $X$ is toroidal, one may assume that $X'$ is toroidal as well. 

\begin{exam} \label{puncturedplane} \rm
Let $G=\sl{2}$ with Borel subgroup $B$ consisting of the upper triangular matrices. 
Consider the spherical homogeneous space $G/H=\aa^2-\{ 0\}$, where 
$
H=\crb{\twomat{1}{*}{0}{1}}
$.
There are two $B$-orbits, an open orbit $O$ and a closed orbit $D$:
$$
O=\crb{\twovec{x}{y}\in \aa^2\,|\, y\neq 0},\quad D=\crb{\twovec{x}{0}\in \aa^2\,|\, x\neq 0}.
$$
We have $\cla{X}\simeq\zz$, namely  the character
$
\chi_n\twomat{a}{b}{0}{a^{-1}}= a^n
$
corresponds to $n\in \zz$.
Moreover, $k(x,y)^{(B)}$ consists of
$y^n$ for $n\in \zz$, up to multiplication by scalars. 
The character associated to $y^n$ is~$\chi_n$. Therefore $\Lambda =\cla{X}$, generated by $y$ or equivalently by $\chi_1$,
and $\mathrm{Q}=\hom(\Lambda,\qq)\simeq \qq$ is spanned by the function 
$\chi^{\ast}:\Lambda\rightarrow \qq$, $\chi^{\ast}(y)=1$.

Consider the following $G$-invariant valuations $v_{min}, v_{max}$ of $k(x,y)$.
If $p\in k[x,y]$ then  $v_{max}(p)=-\deg(p)$ and $v_{min}$ is the minimal degree
of a monomial in $p(x,y)$. 
Since $v_{min}(y)=1$, $v_{max}(y)=-1$,
we have $\cls{V}=\mathrm{Q}$ and 
 $\varrho(v_{min})=-\varrho(v_{max})=\chi^{\ast}$.
 
The $B$-orbit $D$ is a unique color,
so  $\cls{D}=\{D\}$. The valuation $v_D$ measures the order of vanishing  along $D=\{y=0\}$. 
Thus $v_D(y)=1$, hence $\varrho(D)=\chi^{\ast}$.

Let $\cls{R}$ denote the cone in $\mathrm{Q}$ generated by $\chi^{\ast}$, and $-\cls{R}$ the opposite one.
There are four distinct non-trivial colored cones in $\mathrm{Q}$, and six colored fans. These fans are listed in Table \ref{sphembpunctplane}, along with their maximal cones, the corresponding spherical varieties, and their closed $G$-orbits. 
\begin{table}[hbtp]
\centering
\caption{Spherical varieties for the homogeneous space $\aa^2-\{0\}$}
\begin{tabular}{|c| c |c| c|}
\hline
Spherical variety & closed $G$-orbits& maximal colored cones & colored fan \\ 
\hline
\vspace{-10pt} & & & \\
$\aa^2-\{0\}$	& $\aa^2-\{0\}$& $(0,\varnothing)$ & {\includegraphics[width=80pt]{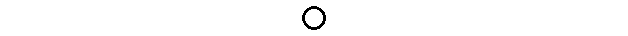}} \\ 
$\aa^2$ 		& $0$ 		& $(\cls{R},\cls{D})$ & {\includegraphics[width=80pt]{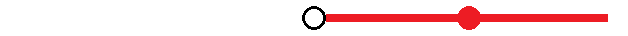}} \\
$\blow _0\aa^2$ 	& $E$ 		& $(\cls{R},\varnothing)$ & {\includegraphics[width=80pt]{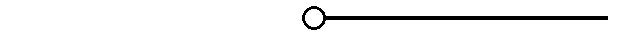}} \\
$\pp^2-\{0\}$	& $W$ 	& $(-\cls{R},\varnothing)$ & {\includegraphics[width=80pt]{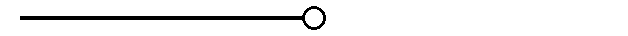}} \\
$\pp^2$		& $W$, $0$& $(\cls{R},\cls{D})$, $(-\cls{R},\varnothing)$ & {\includegraphics[width=80pt]{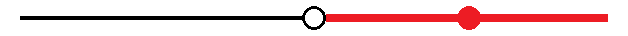}} \\
$\blow_0 \pp^2$	& $W$, $E$ & $(\cls{R},\varnothing)$, $(-\cls{R},\varnothing)$ & {\includegraphics[width=80pt]{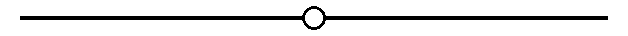}} \\ 
\hline
\end{tabular}
\label{sphembpunctplane}
\end{table}
One can see that the
colored cone $(\cls{R},\cls{D})$ adds a point at the origin, 
$(\cls{R},\varnothing)$ adds the exceptional divisor $E$ of the blowup of the plane at the origin, 
while $-\cls{R}$ adds  ``the line at infinity'' $W$.
The complete spherical varieties $\pp^2$ and $\blow_0\pp^2$ are supported on all of $\cls{V}$.


\end{exam}

\subsection{Spherical Tropicalization} \label{amoeba}

There is a tropicalization map \cite[\S4]{LV}, \cite[\S24]{Ti}
\begin{equation}\label{sfdvsdfvw}
\trpz:\,(G/H)(\overline{K})\rightarrow \cls{V},\quad \gamma\mapsto v_{\gamma}
\end{equation}
defined as follows.
Let 
$f\in k(G/H)$. 
For a sufficiently general $g\in G$, 
the domain of $gf$ contains the image of~$\gamma$
and 
the pullback $\gamma^{\ast}(gf)\in K$ is well-defined.
We have  
$$v_{\gamma}(f)=\{\nu(\gamma^{\ast}(gf))\quad\hbox{\rm for a sufficiently general}\ g\in G\}
=\min_{g\in G} \nu(\gamma^{\ast}(gf)).$$ 
Any $\overline{K}$-point $\gamma:\spec \overline{K}\rightarrow G/H$ factors through $\spec k((t^{1/n}))$ for some $n>0$. 
If $\tilde{\gamma}:\spec k((\tilde{t}))\rightarrow G/H$ is the induced morphism, where $\tilde{t}=t^{1/n}$, we define $v_{\gamma}=v_{\tilde{\gamma}}/n$. This gives a tropicalization map \eqref{sfdvsdfvw}, which is a surjection:
any $G$-invariant valuation of $G/H$ is a scalar multiple of some $v_\gamma$.

Equivalently, let 
$L=k(G)((t))$ 
and consider the standard valuation
$$
\nu:L^{\times}\rightarrow \zz,\quad \sum_n c_nt^n\mapsto \min\{n:c_n\neq 0\}.
$$
Let $\psi_{\gamma}=\mu\circ \phi_{\gamma}$ be the morphism $\spec L\rightarrow G/H$ 
given by the diagram
$$
\begindc{\commdiag}[180]
\obj(2,4)[kg]{$\spec k(G)$} \obj(6,4)[g]{$G$} \obj(2,0)[k]{$\spec K$} \obj(6,0)[x]{$G/H$} \obj(0,2)[l]{$\spec L$} \obj(4,2)[gx]{$G\times G/H$} \obj(8,2)[gh]{$G/H$}
\mor{l}{kg}{} \mor{l}{k}{} \mor{kg}{g}{} \mor{k}{x}{$\gamma$} \mor{gx}{g}{$p_1$} \mor{gx}{x}{$p_2$}[-1,0] \mor{l}{gx}{$\phi_{\gamma}$}[1,1] \mor{gx}{gh}{$\mu$}
\enddc
$$
where 
$\spec k(G)\rightarrow G$ is the generic point and 
$\mu$ is the multiplication map. Then 
$$v_{\gamma}(f)=\nu (\psi_{\gamma}^{\ast}(f)).$$
The extension to $\overline{K}$-points of $G/H$ in straightforward.

\begin{conj}
The map \eqref{sfdvsdfvw} 
extends to a continuous map from the Berkovich analytification $(G/H)^{an}$ to $Q\otimes{\Bbb R}$.
\end{conj}

\begin{defn}
The tropicalization of a closed subvariety $Y\subseteq G/H$ is 
$$\trop Y=\trpz(Y(\overline{K})).$$
\end{defn}


\begin{exam} \label{puncturedplanetrop} \rm
We continue with notation of Example~\ref{puncturedplane} and 
describe 
tropicalizations of curves in $\aa^2-\crb{0}$.
A $\overline{K}$-point $\gamma:\spec \overline{K}\rightarrow \aa^2-\{0\}$
corresponds to a homomorphism of $k$-algebras $\gamma^{\ast}:k[x,y]\rightarrow \overline{K}$ such that not both $x,y$ map to~$0$.
Write $x_{\gamma}, y_{\gamma}\in \overline{K}$ for the images of $x,y\in k[x,y]$.
We claim that
\begin{equation}\label{sdfvfv}
\trpz(\gamma)=c\chi^{\ast},\quad\hbox{\rm where}\ c=\min \crb{\nu (x_{\gamma}),\nu (y_{\gamma})}.
\end{equation}
Indeed, $k(G)=k(g_{ij})$ for $i,j=1,2$ 
and the morphism $\psi_{\gamma}:\spec L\rightarrow \aa^2-\{0\}$ corresponds to the homomorphism of $k$-algebras:
$$
\psi_{\gamma}^{\ast}:k[x,y]\rightarrow L, \quad f(x,y)\mapsto f(g\cdot (x_{\gamma},y_{\gamma})),
$$
where
$$
g\cdot (x_{\gamma},y_{\gamma})=\twomat{g_{11}}{g_{12}}{g_{21}}{g_{22}}\twovec{x_{\gamma}}{y_{\gamma}}
= \twovec{g_{11}x_{\gamma}+g_{12}y_{\gamma}}{g_{21}x_{\gamma}+g_{22}y_{\gamma}}.
$$
The pullback $\psi_{\gamma}^{\ast}(y)$ is then $g_{21}x_{\gamma}+g_{22}y_{\gamma}$. In this expression, no term from $g_{21}x_{\gamma}$ cancels with a term from $g_{22}y_{\gamma}$, since they have distinct coefficients in $k(G)$. Thus
$$
v_{\gamma}(y)=\nu(\psi_{\gamma}^{\ast}(y))=\min \crb{\nu (x_{\gamma}),\nu (y_{\gamma})}.
$$
Next, let $C$ be a curve in $\aa^2-\{0\}$ given by an equation
$$
f(x,y)=\sum_{n,m\ge0}c_{n,m}x^ny^m=0.
$$
A $\overline{K}$-point $\gamma:\spec \overline{K}\rightarrow \aa^2-\{0\}$ factors through $C$ precisely when the kernel 
of~$\gamma^{\ast}$ contains $f(x,y)$, i.e. if $f(x_{\gamma},y_{\gamma})=0$.
Write $f(x,y)=c_{0,0}+f_0(x,y)$. If $c_{0,0}\neq 0$, then $f(x_{\gamma},y_{\gamma})=0$ implies
$$
\min_{(n,m)} \crb{n\nu (x_{\gamma}) +m\nu (y_{\gamma})} \leq 0,
$$
where $(n,m)$ are the pairs of non-negative integers, not both of which are zero, such that $c_{n,m}\neq 0$. 
It is clear that one of $\nu (x_{\gamma})$ and $\nu (y_{\gamma})$ has to be non-positive, hence $v_{\gamma}(y)\leq 0$ and $\trpz(\gamma)=c\chi^{\ast}$ with $c\leq 0$. It follows that $\trop C$ is the ray $-\cls{R}$

\begin{center}
\includegraphics[width=120pt]{fanpzero.png}
\end{center}

\noindent In case $c_{0,0}=0$, there is no restriction on $v_{\gamma}(f)$, and $\trop C$ is all of $\cls{V}=\mathrm{Q}$:

\begin{center}
\includegraphics[width=120pt]{fanblp.png}
\end{center}

\noindent In other words, the tropicalization of a curve passing through the origin of $\aa^2$ 
is all of~$\cls{V}$, while the tropicalization of a curve not passing through it is the ray $-\cls{R}$.
\end{exam}

\subsection{Spherical Tropical Compactifications} \label{tropcompactif}

Let $Y\subseteq G/H$ be a closed subvariety of a spherical homogeneous space. 
Any tropical compactification of $Y$ in a normal variety occurs in a spherical variety. Our goal is to prove Theorem \ref{mainthmspherical}.

\begin{lem} \label{raytrop}\label{workink}\label{kptsameray}
If $v\in \trop Y$, $c\in \qq_{\geq 0}$ then $cv\in \trop Y$, in particular
$$\trop Y=\qq_{\geq 0}\trpz(Y(K)).$$
\end{lem}

\begin{proof}
Let $v=v_{\gamma}\in \trop Y$ for some $\gamma\in Y(\overline{K})$, and pick $c\in \qq_{\geq 0}$. 
There is a morphism $\phi:\spec\overline{K}\rightarrow \spec\overline{K}$ of $k$-schemes given by the endomorphism 
$$
\phi^{\ast}: \overline{K}\rightarrow \overline{K}, \quad f(t)\mapsto f(t^c)
$$
of a $k$-algebra.
Let $\bc{\gamma}=\gamma\circ \phi\in Y(\overline{K})$. 
Let $f\in k(G/H)^{\times}$, and let $g\in G$ be such that the domain of $gf$ contains $\im \gamma=\im \bc{\gamma}$. Then $\bcpb{\gamma}(gf)(t)=\gamma^{\ast}(gf)(t^c)$ and so
$$
\nu (\bcpb{\gamma}(gf))=c\,\nu(\gamma^{\ast}(gf)).
$$
It follows that $v_{\bc{\gamma}}(f)=cv_{\gamma}(f)$, and hence $v_{\bc{\gamma}}=cv$.
\end{proof}

Let $R=k[[t]]$ be the ring of power series over $k$, a discrete valuation ring with field of fractions $K$. If $\gamma$ is a $K$-point of $G/H$ and $G/H\emb X$ a spherical embedding, then due to separatedness there is at most one morphism $\theta:\spec R\rightarrow X$ such that the following diagram commutes:
\begin{equation}\label{werfqer}
\begindc{\commdiag}[40]
\obj(0,8)[lar]{$\spec K$} \obj(0,0)[dvr]{$\spec R$} \obj(16,8)[xp]{$X$}
\mor{lar}{dvr}{} \mor{lar}{xp}{$\gamma$} \mor{dvr}{xp}{$\theta$}[-1,1]
\enddc
\end{equation}
If $\theta$ exists, write $x$ and $\xi$ for the images of the closed and the generic point of $\spec R$, respectively. 
The point $\xi\in G/H$ is the image of $\gamma$. The point $x$ is called the \emph{limit point} of $\gamma$ in $X$, denoted 
$\lim \gamma$. 
We say that $\lim \gamma$ exists in $X$ if  $\theta$ exists. If $X\rightarrow X'$ is a $G$-morphism of spherical varieties
that fixes $G/H$ then the image of $\lim \gamma$ in~$X$ (if it exists) is $\lim \gamma$ in $X'$. This can be extended to $\overline{K}$-points of $G/H$, since any morphism $\spec \overline{K}\rightarrow G/H$ factors through $\spec k((t^{1/n}))$ for some $n\in \zz_{\geq 0}$.


\begin{lem} \label{limpt}
If $\cls{R}\subseteq \cls{V}$ is a ray, $X_R$ the associated toroidal simple spherical variety,
and $v_{\gamma}\in \cls{R}^{\circ}$ for  $\gamma\in G/H(K)$ then 
$\lim \gamma$ exists  and belongs to the closed $G$-orbit~$O$.
\end{lem}

\begin{proof}
The ray $\cls{R}$ is generated by some $G$-invariant discrete valuation $v_D$, associated to a $G$-stable prime divisor $D\subset X_R$ containing $O$. Since $X_R$ is toroidal and $\dim \cls{C}=1$, $O$ is of codimension $1$, hence $D=O$. 
Write $v_{\gamma}=cv_D$ for some $c\in \qq_{> 0}$.

Let $X_R\emb X$ be an equivariant compactification in a toroidal spherical variety. Due to properness of $X'$, there is a (unique) morphism $\theta$ such that \eqref{werfqer} commutes.
Write $\tilde{x}$ and $\tilde{\xi}$ for the closed and the generic point of $\spec R$, respectively, and let $x=\theta(\tilde{x})$ and $\xi=\theta(\tilde{\xi})$ be their images in $X$.
Consider the induced map on stalks:
$$
\theta^{\ast}_x:\cla{O}_{X,x}\rightarrow \cla{O}_{\spec R,\tilde{x}}.
$$
We have $\gamma^{\ast}(f)=\theta_x^{\ast}(f)$ for any $f\in \cla{O}_{X,x}$.
Since $\theta^{\ast}_x$ is a local homomorphism,
$\nu(\gamma^{\ast}(f))=0$ if $f$ does not vanish at $x$, and $\nu(\gamma^{\ast}(f))>0$ otherwise.
We claim that $x\in O$.
Assume the opposite is true. Write $\bar O$ for the closure of $O$ in $X$. 
We consider three  cases: (i) $x\in G/H$, (ii) $x\not\in G/H$ and $x\not \in \bar O$, 
and (iii) $x\in \bar O- O$.
Let $O_x\subseteq X$ be the $G$-orbit containing~$x$.

(i) Pick an affine open set $U=\spec A\subset X_R$ that contains $x$ and intersects~$\bar O$.
It  exists because $X_R$ is quasi-projective. 
Let $\cl{p}\subset A$ be the prime ideal of  $x\in U$ and write $\bar O\cap U=V(\cl{a})$ for some ideal $\cl{a}\subset A$.
Choose $f\in\cl{a}\setminus\cl{p}$.
Then $f$ is a unit in~$A_{\cl{p}}$, hence $\nu(\gamma^{\ast}(f))=0$, which implies $v_{\gamma}(f)\leq 0$. On the other hand,  since $f\in\cl{a}$, $f$ vanishes on $\bar O\cap U$ and so $v_D(f)>0$, which is a contradiction.

(ii) 
Let $D_1,\ldots,D_r$
be $G$-stable prime divisors containing $O_x$.
Since $x\not\in \bar O$, neither of these divisors is equal to $D$.
It follows that the open ray $\cls{R}^{\circ}$ is disjoint from the cone in $\cls{V}$ spanned by $v_{D_1},\ldots,v_{D_r}$,
and therefore there is a function $f\in k(G/H)^{(B)}$ such that $v_{D_i}(f)>0$ for $i=1,\ldots,r$ but
$v_{\gamma}(f)< 0$.
Since $X$ is normal,  $f\in\cla{O}_{X,x'}$ and vanishes at $x'$ for any $x'\in O_x$.
Thus the same is true for $gf$ for any $g\in G$.
It follows that $v_{\gamma}(f)>0$, which is a contradiction.


(iii) 
Let $f\in \cla{O}_{X,x}$ be such that $f$ vanishes along $O_x$ near $x$
bit not along $\bar O$, hence $v_D(f)=0$. On the other hand, $gf$ vanishes at $x$
for  $g$ from an open set of $G$, so $v_{\gamma}(f)>0$, which is a contradiction.
\end{proof}

\begin{lem} \label{ptinboundarycurve}
Let $\cls{C}\subseteq \cls{V}$ be a 
cone and $X$ the associated toroidal simple spherical variety with a closed $G$-orbit $O$.
Let $Y\subseteq G/H$ be a closed subvariety, $\overline{Y}\subseteq X$ its closure, and $x\in O\cap \overline{Y}$.
Then there is $\gamma\in Y(K)$ such that $\lim\gamma=x$ and $v_{\gamma}\in\cls{C}^{\circ}$.
\end{lem}

\begin{proof}
By cutting $Y\subset X$ with general hypersurfaces passing through $x$, we can assume without loss of generality that $Y$
is $1$-dimensional.
Let $B\subset k(Y)$ be a discrete valuation ring that dominates $\cla{O}_{\overline{Y},x}$.
By the Cohen's Structure Theorem \cite[Prop. 10.16]{E}), 
the completion $\widehat{B}$ of $B$ is isomorphic to $R=k[[t]]$, and we identify it with this ring.
The inclusions
\begin{equation*} \label{gammamaps}
\cla{O}_{\overline{Y},x} \emb B\emb \widehat{B} \emb K
\end{equation*}
give rise to a morphism 
$\theta:\spec\widehat{B}\rightarrow \overline{Y}$ of schemes over $k$
that sends the generic point of $\spec\widehat{B}$ to the generic point of $Y$ and the closed point to $x$.

The restriction of $\theta$ to $\spec K$ is a $K$-point $\gamma$ of $Y$ such that
$\lim \gamma$ exists and is equal to $x$.
Now we show that $v_{\gamma}\in \cls{C}^{\circ}$. Let $\cls{C}_0$ be the ray generated by $v_{\gamma}$ in $\cls{V}$.
If $\cls{C}_0$ is in $\cls{C}^{\circ}$ then we are done. Assume not. We consider two cases, (i) $\cls{C}_0$ is not contained in $\cls{C}$, and (ii) $\cls{C}_0$ is in $\cls{C}-\cls{C}^{\circ}$.
Let $X_0$ be the associated toroidal simple spherical variety of $\cls{C}_0$
with closed $G$-orbit~$O_0$. Write $Y_0\subseteq X_0$ for the closure of $Y$.

(i) Let $\cla{F}$ be a fan (without colors) with cones $\cls{C}_0$ and $\cls{C}$, and let $X'$ be the associated toroidal spherical variety. There are open $G$-embeddings $X_0\emb X'$ and $X\emb X'$. 
Since the cones $\cls{C}_0$ and $\cls{C}$ don't intersect (except at the origin), the orbits $O_0$ and $O$ are disjoint in $X'$. 
Since $v_{\gamma}\in \cls{C}_0^{\circ}$, $\lim \gamma$ exists in $X_0$ and is in $O_0$ by Lemma \ref{limpt}.
But this can't be true because $\lim \gamma=x\in O$ as shown above.

(ii) There is a birational $G$-morphism $X_0\rightarrow X$. 
Since $\cls{C}_0$ is not contained in~$\cls{C}^{\circ}$, the closed orbit $O_0$ of $X_0$ does not map to the closed orbit $O$ of $X$.
From Lemma~\ref{limpt}, $\lim \gamma$ exists in $X_0$ and is in $O_0$. 
The image of $\lim \gamma$ under $X_0\rightarrow X$, which is the limit point of $\gamma$ in $X$,
lies in the orbit $O$, which is a contradiction.
\end{proof}

The next proposition is an extension of \cite[Lemma 2.2]{Te} 
to spherical varieties. 

\begin{prop} \label{tevelevlem}
Let $X$ be a simple toroidal spherical variety with a closed $G$-orbit~$O$, and let $\cls{C}$ be the associated cone in $\mathrm{Q}$. Then $\trop Y$ intersects the relative interior of $\cls{C}$ if and only if the closure $\overline{Y}\subseteq X$ intersects the closed orbit $O$.
\end{prop}
\begin{proof}
First assume that $\trop Y\cap \cls{C}^{\circ}\neq \varnothing$, and let $v\in \trop Y\cap \cls{C}^{\circ}$. By Lemma~\ref{kptsameray}, 
we may assume that $v=v_{\gamma}$ with $\gamma\in Y(K)$. Let $X_0$ be the toroidal simple spherical variety associated to $\cls{C}_0$, $O_0$ the closed $G$-orbit of $X_0$, and $Y_0\subseteq X_0$ the closure of~$Y$. 
By~Lemma \ref{limpt}, the limit point of $\gamma$ in $X_0$ is in $O_0$,
which shows that $Y_0\cap O_0$ is non-empty. 
Since $\cls{C}_0$ is in $\cls{C}^{\circ}$, there is a birational $G$-morphism $f:X_0\rightarrow X$ 
that sends $O_0$ to~$O$ and maps $Y_0$ to $\overline{Y}$. 
Therefore, $\overline{Y}\cap O\neq \varnothing$.

Now assume that $\overline{Y}\cap O\neq \varnothing$. 
By Lemma \ref{ptinboundarycurve}, 
there is a $K$-point $\gamma\in Y(K)$ such that $v_{\gamma}\in \cls{C}^{\circ}$, so that
$\trop Y\cap \cls{C}^{\circ}\neq \varnothing$. This completes the proof.
\end{proof}

The following propositions and their proofs generalize \cite[Prop. 2.3 and 2.5]{Te}.

\begin{prop} \label{propertrop}
Let $X$ be a toroidal spherical variety associated with a fan $\cla{F}$.
Then $\overline{Y}$ is complete if and only if $\trop Y\subseteq \supp \cla{F}$.
\end{prop}

\begin{proof}
Suppose that $\overline{Y}$ is complete but $\trop Y$ is not contained in $\supp\cla{F}$. 
Let $X\emb X'$ be an  equivariant compactification of $X$ in some toroidal spherical variety~$X'$, associated to a fan $\cla{F}'$ containing $\cla{F}$. Since $X'$ is complete, $\supp \cla{F}'=\cls{V}$, hence there is a cone $\cls{C}$ of $\cla{F}'$ whose interior does not intersect $\cla{F}$ and contains a point of $\trop Y$. Let $Y'$ be the closure of $Y$ in $X'$. 
Since $\overline{Y}$ is complete, $Y'=\overline{Y}$. Thus $Y'$ does not intersect the boundary $X'-X$. This boundary contains the closed $G$-orbit corresponding to $\cls{C}$, and this contradicts Proposition \ref{tevelevlem}.

Now assume that $\trop Y\subseteq \supp\cla{F}$ but $\overline{Y}$ is not complete. Let $X\emb X'$, $\cla{F}'$, and $Y'$ be as above. Since $\overline{Y}$ is not complete but $Y'$ is, as a closed subvariety of a complete variety, the inclusion $\overline{Y}\subset Y'$ is strict. In particular, $Y'$ intersects some $G$-orbit in $X'-X$, which corresponds to a cone $\cls{C}$ of $\cla{F}'$ whose interior does not intersect $\cla{F}$. By Proposition \ref{tevelevlem}, $\cls{C}^{\circ}$ intersects $\trop Y$, but this is not the case as the latter is contained in $\supp \cla{F}$.
\end{proof}

\begin{prop} \label{tropsupport}
If $\overline{Y}$ is a tropical compactification of $Y$ in a toroidal spherical variety $X$ associated to a fan $\cla{F}$
then $\supp \cla{F}=\trop Y$.
\end{prop}

\begin{proof}
Suppose $\supp \cla{F}\ne\trop Y$.
By Proposition \ref{propertrop}, $\supp \cla{F}$ contains $\trop Y$. Let $v\in\supp\cla{F}$ be an element not in $\trop Y$. Then the entire ray $R$ generated by~$v$ is not in $\trop Y$ (Lemma \ref{raytrop}). 
Let $\cla{F}'$ be a refinement of $\cla{F}$ that contains the cone~$R$ 
and let $X'$ be the toroidal spherical variety defined by it.
There is a proper birational $G$-morphism $f:X'\rightarrow X$. 
The closure $Y'\subseteq X'$ of $Y$, which is the pure transform of $\overline{Y}$ with respect to $f$, is a tropical compactification of $Y$ (Prop. \ref{tropcompbir}). The multiplication morphism $G\times Y'\rightarrow X'$ is faithfully flat, hence surjective, and so $Y'$ intersects every $G$-orbit of $X'$. But by Proposition \ref{tevelevlem} this is not the case for the closed $G$-orbit associated to the cone $R$.
\end{proof}

\begin{proof}[Proof of Theorem \ref{mainthmspherical}]
By Theorem \ref{tropcompexist}, tropical compactifications of $Y$  exist. Let $\overline{Y}\subseteq X$ be one of them
 and let $\cla{F}$ be the fan associated to $X$. Let $\cla{F}'$ be the fan obtained by removing all colors from $\cla{F}$, i.e. $\cla{F}'$ consists of all cones $\cls{C}\cap \cls{V}$ for $(\cls{C},\cls{F})\in\cla{F}$, and let $X'$ be the associated toroidal spherical variety. In particular, $\supp\cla{F}'=\supp\cla{F}$. There is a proper birational $G$-morphism $f:X'\rightarrow X$ restricting to the identity on $G/H$. 
By Prop. \ref{tropcompbir}, the closure $Y'\subseteq X'$ of $Y$ is a tropical compactification in a toroidal spherical variety.
Finally, 
$$
\supp \cla{F}=\supp\cla{F}'=\trop Y
$$
by Prop.~\ref{tropsupport}, which completes the proof.
\end{proof}

\begin{exam}\rm
In the notation of Example~\ref{puncturedplanetrop},
let $\overline{C}\subseteq \mathbb{X}$ be a tropical compactification with $\mathbb{X}$ toroidal. 
Then $\mathbb{X}$ is $\pp^2-\crb{0}$ if $C$ does not pass through the origin $O\in\aa^2$
or $\blow_0\pp^2$ if it does. In particular, $\mathbb{X}$ contains the boundary $W$ and the intersection number $\overline{C}\cdot W$ is the degree of the curve $C$. On the other hand, if $\mathbb{X}=\blow_0\pp^2$, the intersection number $\overline{C}\cdot E$
is the multiplicity of $C$ at the origin, which is typically smaller than $\deg C$. 
It follows that even though we can define multiplicities of cones of $\trop X$  as in the toric case \cite{ST}, 
and they encode important geometric information, the balancing condition does not hold.
\end{exam}

\section{Further Examples of Spherical Tropicalization} \label{allexamples}

We use notation of \S\ref{sphericaltrop}.
Lemma~\ref{kptsameray} allows us to use $K$-points instead of $\overline{K}$-points. 

\subsection{Subvarieties of $\gln$} \label{lingps}
We will prove Theorem \ref{tropgln} and  its analog for $\sln$:

\begin{thm} \label{tropsln}
Let $Y=(f_1=\ldots=f_s=0)\subset \sln$ be a closed subvariety.
Then $\trop Y$ consists of $(n-1)$-tuples $(\alpha_1,\ldots,\alpha_{n-1})$ of  
invariant factors (in~decreasing order) of matrices $M\in\sln(\overline{K})$
such that $f_1(M)=\ldots =f_r(M)=0$.
\end{thm}

We consider 
$G=\gln$ or $\sln$ simultaneously. 
Consider the group $G\times G$ with Borel subgroup $B$ consisting of pairs of an upper and a lower triangular matrices. 
Let $H=\{(g,g)\in G\times G\}$ be the diagonal. 
Then $(G\times G)/H\simeq G$ 
with the action given by left and right multiplication
$
(g,h)\cdot x=gxh^{-1}$ for $(g,h)\in G\times G$, $x\in G$.
This is a spherical homogeneous space by the Bruhat decomposition.
Let  
$x_{ij}$ be matrix coordinates for $G$.
For $G=\gln$, the group of characters is $\cla{X}\simeq \zz^{2n}$, 
where $(\mathbf{l},\mathbf{m})=(l_1,\ldots,l_n,m_1,\ldots,m_n)\in \zz^{2n}$ corresponds to
\begin{equation}\label{weevwev}
\chi_{(\mathbf{l},\mathbf{m})}: B\rightarrow k^{\times},\quad \brk{(a_{ij}),(b_{ij})}\mapsto \prod_{i=1}^n a_{ii}^{-l_i}b_{ii}^{m_i}.
\end{equation}
For $G=\sln$,  $\cla{X}\cong \zz^{2(n-1)}$, where $(\mathbf{l},\mathbf{m})=(l_1,\ldots,l_{n-1},m_1,\ldots,m_{n-1})\in \zz^{2(n-1)}$ 
corresponds to the character \eqref{weevwev} with $l_n=m_n=0$.

For $G=\gln$, the lattice $\Lambda$ is generated by  $B$-semi-invariant functions
$$
f'_i=\det \brk{
\begin{array}{cccc}
x_{i,i} & x_{i,i+1} & \ldots & x_{i,n} \\
x_{i+1,i} & x_{i+1,i+1} & \ldots & x_{i+1,n} \\
\vdots & \vdots & \ddots & \vdots \\
x_{n,i} & x_{n,i+1} & \ldots & x_{n,n} \\
\end{array}}$$
for $i=1,\ldots,n$.
For example, $f'_1=\det x$ and $f'_n=x_{nn}$. The character of $f'_i$ is $\chi'_i=\chi_{(\mathbf{m}_i,\mathbf{m}_i)}$, where $\mathbf{m}_i=(0,\ldots,0,1,\ldots,1)$ (the first entry~$1$ is the $i$-th~one). 
The colors are the $B$-stable prime divisors $D_2,\ldots, D_n$ given by the functions $f'_2,\ldots,f'_{n}$, and $\varrho(D_i)=(\chi'_i)^{\ast}$ in~$\mathrm{Q}$ (the dual basis to $\chi'_i$). 
A better set of generators is $f_1,\ldots,f_n$, where $f_i=f'_i/f'_{i+1}$ for $i<n$, and $f_n=f'_n$. The character associated to $f_i$ is $\chi_i=\chi_{(\mathbf{e}_i,\mathbf{e}_i)}$, 
where $\mathbf{e}_i=(0,\ldots,0,1,0,\ldots,0)$ (the $1$ in the $i$-th entry). The vector space $\mathrm{Q}$ is $n$-dimensional, spanned by the dual basis $\chi_1^{\ast},\ldots,\chi_n^{\ast}$ to $f_1,\ldots,f_n$.

For $G=\sln$, we define $f'_2,\ldots, f'_n$ as for $G=\gln$ 
and then let $f_1=1/f'_{2}$, and $f_i=f'_i/f'_{i+1}$ for $i=2,\ldots, n-1$. The character $\chi'_i$ associated to $f'_i$ is $\chi_{(\mathbf{l},\mathbf{l})}$, where $\mathbf{l}=(-1,\ldots,-1,0,\ldots,0)$ (the first zero in the $i$-th position). The character $\chi_i$ associated to $f_i$ is as in the case of $G=\gln$. The vector space $\mathrm{Q}$ is $(n-1)$-dimensional, spanned by by the dual basis $\chi_1^{\ast},\ldots,\chi_{n-1}^{\ast}$.

We now construct the tropicalization map $\trpz:G(\overline{K})\rightarrow \mathrm{Q}$. Let $\gamma\in G(\overline{K})$, and write $\gamma^{\ast}:k[G]\rightarrow \overline{K}$ for the associated homomorphism of $k$-algebras, where 
$$k[\gln]=k[x_{ij}]_{\det x}\quad\hbox{\rm and}\quad k[\sln]=k[x_{ij}]/(1-\det x).$$ 
Write $x_{\gamma}$ for the matrix $\gamma^\ast(x)$. 
Let $L=\bigcup_{m}k(G\times G)((t^{1/m}))$ as in  \S\ref{amoeba}.
The morphism $\psi_{\gamma}:\spec L\rightarrow G$ is induced by the map
$$
\psi_{\gamma}^{\ast}:k[G]\rightarrow L, \quad f(x)\mapsto f(gx_{\gamma}h^{-1}).
$$
Thus $v_{\gamma}(f'_i)$ is the smallest value of the valuations of $i\times i$ minors of the matrix~$x_{\gamma}$ and $v_{\gamma}(f_i)=v(f'_i)-v(f'_{i+1})$ with the special cases $v(f_n)=v(f'_n)$ if $G=\gln$, or $v(f_1)=-v(f'_2)$ if $G=\sln$. This is a well-known method for calculating the invariant factors of a matrix, hence $v(f_i)=\alpha_i$,
where $\alpha_1,\ldots,\alpha_n$ are invariant factors of the matrix $x_{\gamma}$ in the decreasing order.
Therefore, 
$$
\trpz(\gamma)= \alpha_1\chi_1^{\ast}+\cdots+\alpha_n\chi_n^{\ast}\quad\text{in }\mathrm{Q},
$$
where $\alpha_n=0$ when $G=\sln$.
This proves Theorems \ref{tropgln} and \ref{tropsln}.

If $G=\gln$, the valuation cone, which is the image of $\trpz$, is the set
$$
\cls{V}=\{(\alpha_1,\ldots,\alpha_n)\in \mathrm{Q}:\alpha_1 \geq \cdots \geq \alpha_n\}.
$$
while if $G=\sln$, it is the set
$$
\cls{V}=\crb{(\alpha_1,\ldots,\alpha_{n-1})\in \mathrm{Q}:\alpha_1 \geq \cdots \geq \alpha_{n-1}\text{ and } \sum_{i=1}^{n-1}\alpha_i+\alpha_{n-1}\geq 0}.
$$
Indeed, the sum $\sum\limits_{i=1}^{n-1}\alpha_i$ for a matrix of determinant~$1$ is equal to 
$-\alpha_{n-2}\ge-\alpha_{n-1}$. The valuation cones of $\gl{2}$ and $\sl{3}$ are shaded gray areas in Figures \ref{valconegl2} and \ref{valconesl3}.

\begin{figure}
    \centering
\caption{Valuation cones of $\gl{2}$ and $\sl{3}$}
    \begin{subfigure}[b]{0.45\textwidth}
\includegraphics[width=\textwidth]{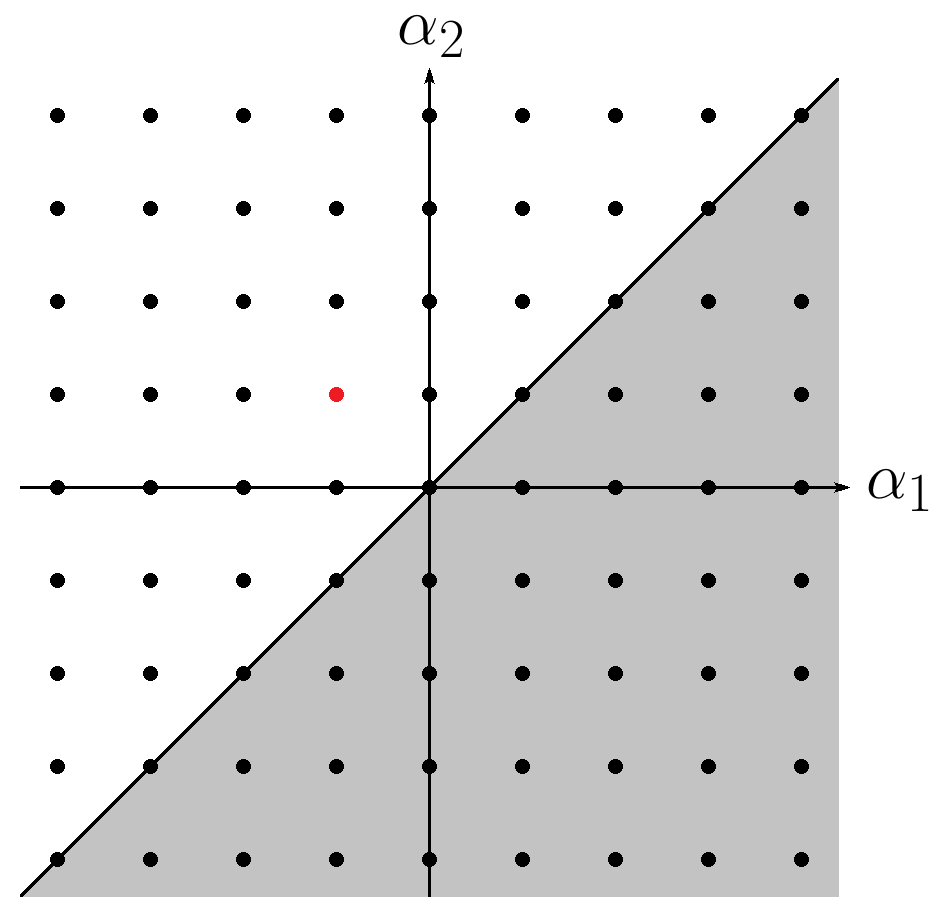}
        \caption{Valuation cone of $\gl{2}$}
        \label{valconegl2}
    \end{subfigure}
    \quad 
    \begin{subfigure}[b]{0.45\textwidth}
        \includegraphics[width=\textwidth]{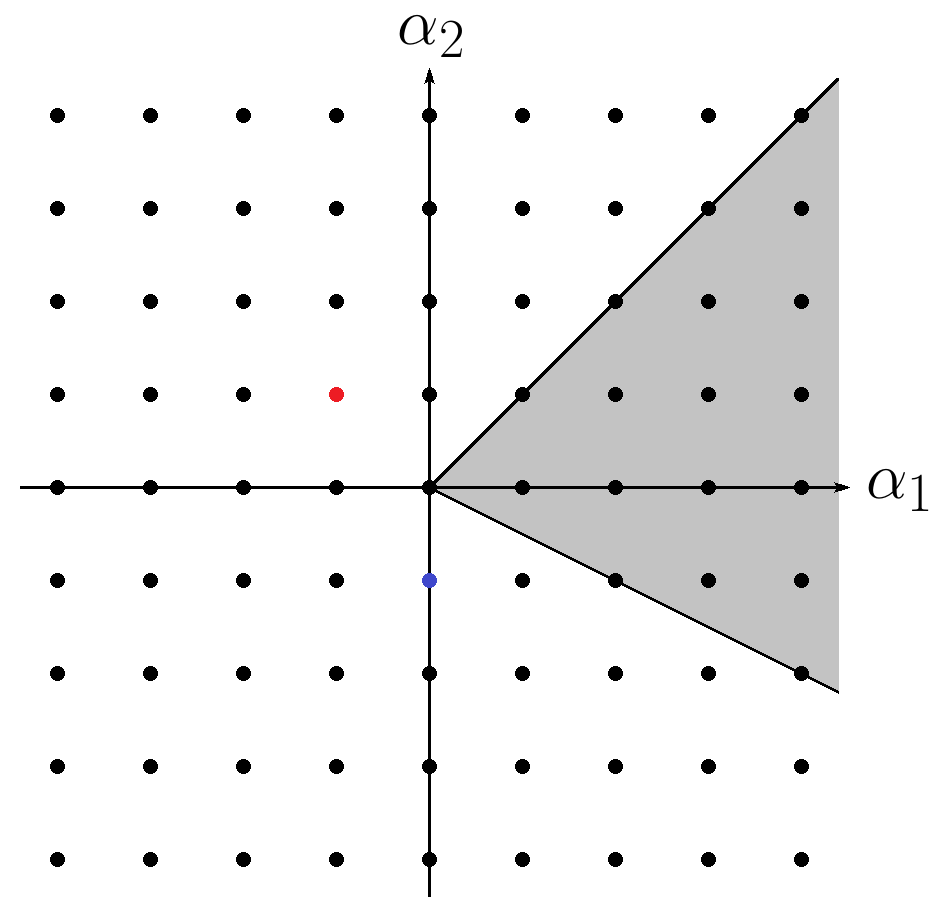}
        \caption{Valuation cone of $\sl{3}$}
        \label{valconesl3}
    \end{subfigure}
\end{figure}

\begin{exam} \label{tropcurvegl2} \rm
Let $C$ be the line in $\gl{2}$ defined by the ideal
$$
I=(x_{11}-x_{12}-1,\,x_{12}-x_{21},\,x_{22}).
$$
A matrix in $C(K)$ 
is of the form
$$
\twomat{z(t)+1}{z(t)}{z(t)}{0},\quad z(t)\in K.
$$
The determinant of this matrix is $-z(t)^2$. If $\nu(z(t))\geq 0$, then the smallest invariant factor is $\alpha_2=0$, and $\alpha_1=\nu(-z(t)^2)$, so that $\alpha_1$ can be any non-negative integer. This gives the ray consisting of the positive $\alpha_1$-axis in $\mathrm{Q}$, say $\cls{R}_1$. If $\nu(z(t))< 0$, then the smallest invariant factor is $\alpha_2=\nu(z(t))$, and $\alpha_1+\alpha_2=\nu(-z(t)^2)=2\alpha_2$, so that $\alpha_1=\alpha_2$. This corresponds to the ray along the line $\alpha_1=\alpha_2$ in the third quadrant, call it $\cls{R}_2$. 
Thus $\trop C$ is the union of rays $\cls{R}_1$ and $\cls{R}_2$, see Figure~\ref{examtropcurve}.

Since $\trop C$ has dimension $1$, it completely determines the toroidal spherical variety in which the tropical compactification occurs. 
We describe this spherical variety.
We view $\gl{2}$ as a quasi-affine variety in $\aa^4$ (with coordinates $x_{ij}$). Consider the projective space $\pp^4$ with homogeneous coordinates
$$
\brk{X_0,X}=\brk{X_0,\twomat{X_{11}}{X_{12}}{X_{21}}{X_{22}}}.
$$
We identify $\aa^4$ with the affine space $(X_0\neq 0)$ in $\pp^4$. The action of $\gl{2}\times \gl{2}$ on $\gl{2}$ extends to an action on all of $\pp^4$:
$$
(g,h)\cdot  (X_0,X)=(X_0,gXh^{-1}),\quad (g,h)\in \gl{2}\times \gl{2}, \, (X,X_0)\in \pp^4,
$$
and so $\gl{2}\emb \pp^4$ is a spherical embedding. Its colored fan is given in Figure~\ref{wefvwefv}(A). 
The rays $\cls{R}_1$ and $\cls{R}_2$ are cones of this fan, and so 
the spherical variety associated to $\trop C$ is a $\gl{2}$-stable open subvariety of $\pp^4$.
It follows that the tropical compactification of $C$ is its closure in 
$\pp^4$, which contains two points in the boundary, 
$$
\brk{1,\twomat{1}{0}{0}{0}}\quad\text{and}\quad \brk{0,\twomat{1}{1}{1}{0}}.
$$

\begin{figure}
    \centering
        \caption{$\trop C$ is the union of two rays}
        \label{examtropcurve}
\includegraphics[width=0.45\textwidth]{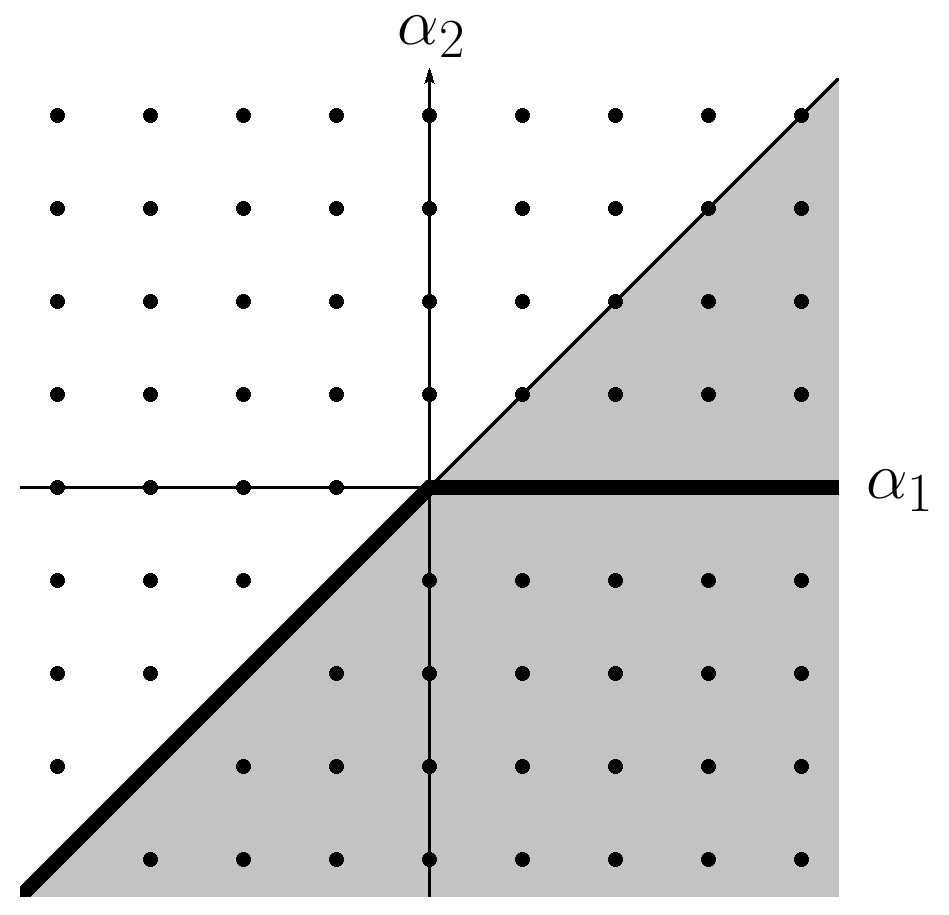}
\end{figure}

\end{exam}

\begin{exam} \label{examplesgl2} \rm
Let $Y_1$ be the hyperplane $(x_{11}=1)$ in $\gl{2}$. A matrix in $Y_1(K)$ 
satisfies $\nu(x_{11}(t))=0$, hence the smallest invariant factor $\alpha_2$ is  non-positive. There is no restriction on $\alpha_1$. Indeed, if $(\alpha_1,\alpha_2)$ is a pair of integers with $\alpha_1\geq \alpha_2$ and $\alpha_2\leq 0$, then  the  matrix
$$
\twomat{1}{t^{\alpha_1}}{t^{\alpha_2}}{0}\in Y_1(K)
$$
has invariant factors $(\alpha_1,\alpha_2)$. Therefore, $\trop Y_1$ is the one of Figure~\ref{afvqvq}(A).

If $Y_2=V(x_{21}-x_{12}^2)$, then $\trop Y_2$ is the whole valuation cone. Indeed, for any pair of integers $(\alpha_1,\alpha_2)$ with $\alpha_1\geq \alpha_2$, the matrix
$$
\twomat{t^{\alpha_1}}{0}{0}{t^{\alpha_2}}\in Y_2(K)
$$
has invariant factors $(\alpha_1,\alpha_2)$.
Now consider the subvariety $Y=V(x_{11}-1,x_{21}-x_{12}^2)$. An matrix in $Y(K)$
is of the form
$$
\twomat{1}{y(t)}{y^2(t)}{z(t)},\quad y(t),z(t)\in K.
$$
The determinant of this matrix is $z(t)-y^3(t)$. If $\nu(y(t)),\nu(z(t))\geq 0$, then $\alpha_2=0$ and $\alpha_1$ can be any positive number, which gives the positive $\alpha_1$-axis.
If $\nu(y(t))\leq 0,\nu(z(t))\geq 0$, then $\alpha_2=2\nu(y(t))$ and $\alpha_1=\nu(y(t))$. This is the ray along the line $\alpha_2=-\alpha_1/2$ in the third quadrant.
If $\nu(y(t))\geq 0,\nu(z(t))\leq 0$, then $\alpha_2=\nu(z(t))$ and $\alpha_1=0$, which is the negative $\alpha_2$-axis.
If $\nu(y(t)),\nu(z(t))\leq 0$, then there are three subcases. If $\nu(z(t))$ is more than $2\nu(y(t))$ or less than $3\nu(y(t))$, then we get back the ray along $\alpha_2=-\alpha_1/2$ or the negative $\alpha_2$-axis, respectively. If $3\nu(y(t))\leq \nu(z(t))\leq 2\nu(y(t))$, then $\alpha_2=\nu(z(t))$ and $\alpha_2/2\leq \alpha_1\leq 0$, and we get the cone between $\alpha_2=-\alpha_1/2$ and the negative $\alpha_2$-axis.
$\trop Y$ is given in Figure \ref{afvqvq}(B). Even though $Y=Y_1\cap Y_2$, 
$\trop Y$ is strictly smaller than  $\trop Y_1 \cap \trop Y_2$.

\begin{figure}
	\caption{Tropicalizations of subvarieties of $\gl{2}$}\label{afvqvq}
    \centering
    \begin{subfigure}[b]{0.45\textwidth}
        \includegraphics[width=\textwidth]{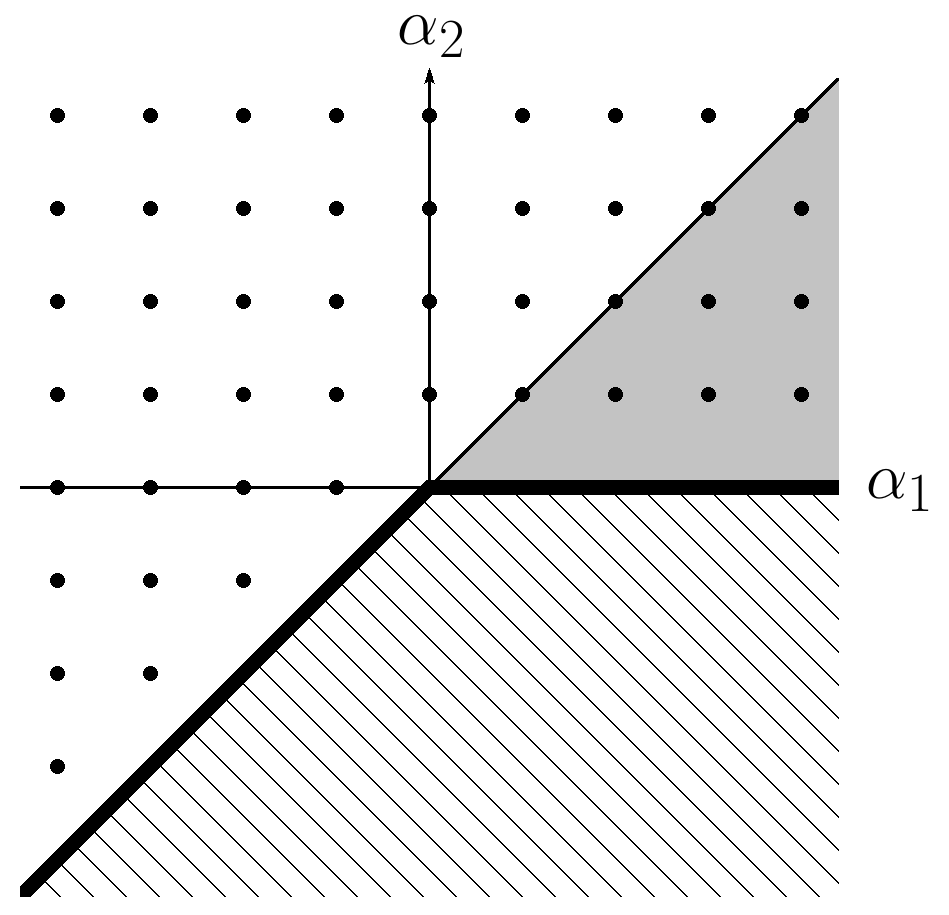}
	\caption{$\trop V(x_{11}-1)$}
        \label{examtrophyperp}
    \end{subfigure}
    \quad
    \begin{subfigure}[b]{0.45\textwidth}
	\includegraphics[width=\textwidth]{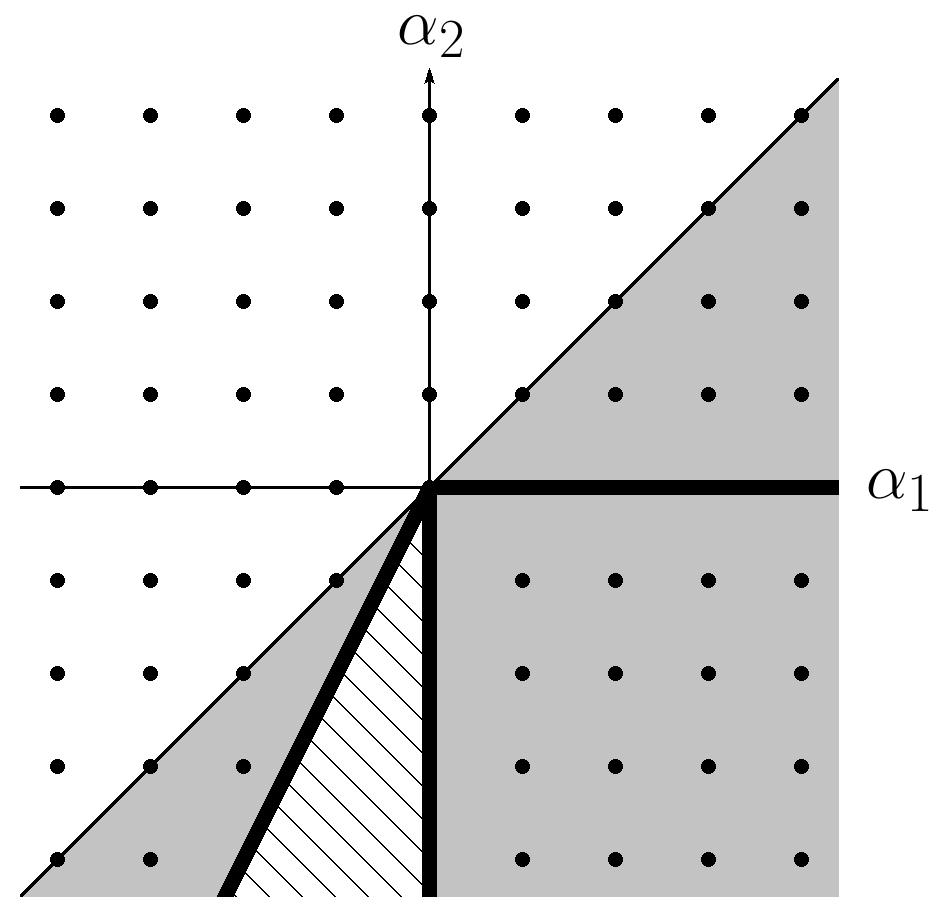}
\caption{$\trop V(x_{11}-1,x_{21}-x_{12}^2)$}
        \label{examtropsecond}
    \end{subfigure}
\end{figure}

\end{exam}

\begin{exam} \label{examtropso4}\rm
Consider the special orthogonal group $\so{4}$ as a subvariety of $\sl{4}$:
$$
\so{4}=\crb{x\in \sl{4}:x^tx=e},
$$
where $e$ is the identity matrix. Let $x(t)\in\so{4}(K)$.
 The invariant factors of $x(t)^t$ and $x(t)$ are the same, while the ones of $e$ are all zero. The invariant factors $(\alpha_1,\alpha_2,\alpha_3,\alpha_4)$ of $x(t)$ must satisfy the following Horn's inequalities (see \S\ref{represvarty}):
$$
\alpha_1+\alpha_4\geq 0\quad \text{and}\quad \alpha_2+\alpha_3\geq 0.
$$
Since $x(t)$ has determinant $1$, $\alpha_4=-\alpha_1-\alpha_2-\alpha_3$, and the first inequality becomes $\alpha_2+\alpha_3\leq 0$, hence $\alpha_3=-\alpha_2$. This forces $\alpha_4=-\alpha_1$. We show that any quadruple $(\alpha_1,\alpha_2,\alpha_3,\alpha_4)$ that satisfies these two conditions is in $\trop \so{4}$.

Pick $(\alpha_1,\alpha_2,-\alpha_2,-\alpha_1)$ with $\alpha_1\geq \alpha_2\geq 0$. The matrix with entries in $\overline{K}$:
$$
\brk{ \begin{array}{cccc}
t^{-\alpha_1} & \sqrt{1-t^{-2\alpha_1}} & 0 & 0 \\
-\sqrt{1-t^{-2\alpha_1}} & t^{-\alpha_1} & 0 & 0 \\
0 & 0 & t^{-\alpha_2} & \sqrt{1-t^{-2\alpha_2}} \\
0 & 0 & -\sqrt{1-t^{-2\alpha_2}} & t^{-\alpha_2} \\
\end{array}}
$$
 is orthogonal, of determinant $1$, and has invariant factors $(\alpha_1,\alpha_2,-\alpha_2,-\alpha_1)$. It follows that
$
\trop \so{4}=\crb{(\alpha_1,\alpha_2,\alpha_3)\in \cls{V}:\alpha_3=-\alpha_2}$.

\end{exam}

\subsection{Subvarieties of $\pgln$} \label{pgln}
Any $x\in \pgln(\overline{K})$ can be represented by a matrix 
for which the smallest invariant factor is $0$. Any such representation has the same invariant factors,
which we call 
(without the last $0$) the 
\emph{invariant factors} of~$x$. 
The proof of the next theorem is the same as for Theorems \ref{tropgln} and~\ref{tropsln}.

\begin{thm} \label{troppgln}
Let $Y\subset \pgln$ be a closed subvariety. 
Then $\trop Y$ consists of $(n-1)$-tuples $(\alpha_1,\ldots,\alpha_{n-1})$ of  invariant factors 
of elements $x\in Y(\overline{K})$.
\end{thm}

The valuation map assigns to $x$ its invariant factors and
the valuation cone is 
$$
\cls{V}=\{(\alpha_1,\ldots,\alpha_{n-1})\in \mathrm{Q}:\alpha_1 \geq \cdots \geq \alpha_{n-1}\geq 0\}.
$$

\subsection{Tropicalization of the Representation Variety of $\pi_1(\mathbb{S}_{0,3})$} \label{represvarty}

Let $\mathbb{S}_{0,3}$ denote the Riemann sphere with $3$-punctures. Its fundamental group  is 
$$
\Gamma=\pi_1(\mathbb{S}_{0,3})=\inn{a,b,c:abc=1}=\inn{a,b,c:ab=c^{-1}},
$$
where $a,b,c$ are loops around the first, second, and third puncture, respectively. $\Gamma$
is isomorphic to the free group in $2$ generators but this presentation is more natural for the problem.
Let $G$ be $\gln$ or $\sln$. The $G$-representation variety of $\Gamma$ is
$$
\mathrm{Rep}_{G}(\Gamma)=\hom(\Gamma,G)=\crb{(x,y,z)\in G^3:xy=z^{-1}}.
$$
We view $G^3$ as a homogeneous space via the action of $G^6=(G\times G)^3$ by multiplication on the left and on the right. The lattice $\mathrm{Q}$ has dimension $3n$ if $G=\gln$, and $3(n-1)$ if $G=\sln$. Consider the standard basis for $\mathrm{Q}$, which is an extension of the one given in \S\ref{lingps} to the product of three copies of $G$.

If $G=\gln$, the set $\trop \mathrm{Rep}_{G}(\pi_1(\mathbb{S}_{0,3}))$  consists of (positive scalar multiples of)  $(3n)$-tuples of integers $(\alpha_1,\ldots,\alpha_n,\beta_1,\ldots,\beta_n,\gamma_1,\ldots,\gamma_n)$ that appear as invariant factors of matrices $x,y,z$ with entries in $K$, such that $xy=z^{-1}$. We write $(\gamma'_1,\ldots,\gamma'_n)$ for the invariant factors of the matrix $z^{-1}$, which are $\gamma'_{1}=-\gamma_n$, $\gamma'_2=-\gamma_{n-2}$, etc. It suffices to describe $(3n)$-tuplets of integers 
$$(\alpha_1,\ldots,\alpha_n,\beta_1,\ldots,\beta_n,\gamma'_1,\ldots,\gamma'_n)$$ that appear as invariant factors of matrices $x,y,z'$ with entries in $K$, such that $xy=z'$. The case of $G=\sln$ is similar.

This problem is equivalent to Horn's problem, see \cite[Thm. 7 \& 17]{F}. Its
solution is given by the Horn's inequalities
\begin{equation} \label{hornequality}
\sum_{i=1}^n \alpha_i+\sum_{i=1}^n \beta_i=\sum_{i=1}^n \gamma'_i,
\end{equation}
and 
\begin{equation} \label{horninequality}
\sum_{k\in K} \gamma'_i\leq \sum_{i\in I} \alpha_i+\sum_{j\in J} \beta_i\quad\text{for all }(I,J,K)\in T_r^n,
\end{equation}
where $I,J,K$ are subsets of $\crb{1,\ldots,n}$ of the same cardinality, and $T_r^n$ are defined inductively as
$$
T_r^n=\crb{(I,J,K)\in U_r^n:\begin{array}{l}\smallskip\text{for every }p<r\text{ and } (F,G,H)\in T_p^r, \\ \displaystyle{\sum_{f\in F}i_f+\sum_{g\in G}j_g\leq \sum_{h\in H}k_h+p(p+1)/2} \\  \end{array}}
$$
where $U_r^n$ are the sets of triplets $(I,J,K)$ given by:
$$
U_r^n=\crb{(I,J,K):\sum_{i\in I}i+\sum_{j\in J}j=\sum_{k\in K}k+r(r+1)/2}.
$$

For example,  $\trop\mathrm{Rep}_{\sl{2}}(\Gamma)$ is given by the inequalities
$$
\alpha_1\leq \beta_1+\gamma_1,\quad \beta_1\leq \gamma_1+\alpha_1,\quad \gamma_1\leq \alpha_1+\beta_1,
$$
see Figure \ref{troprepvar}. 
\begin{figure}
\caption{Tropicalization of the $\sl{2}$-representation variety of $\pi_1(\mathbb{S}_{0,3})$}
\label{troprepvar}
\centering
\includegraphics[width=0.45\textwidth]{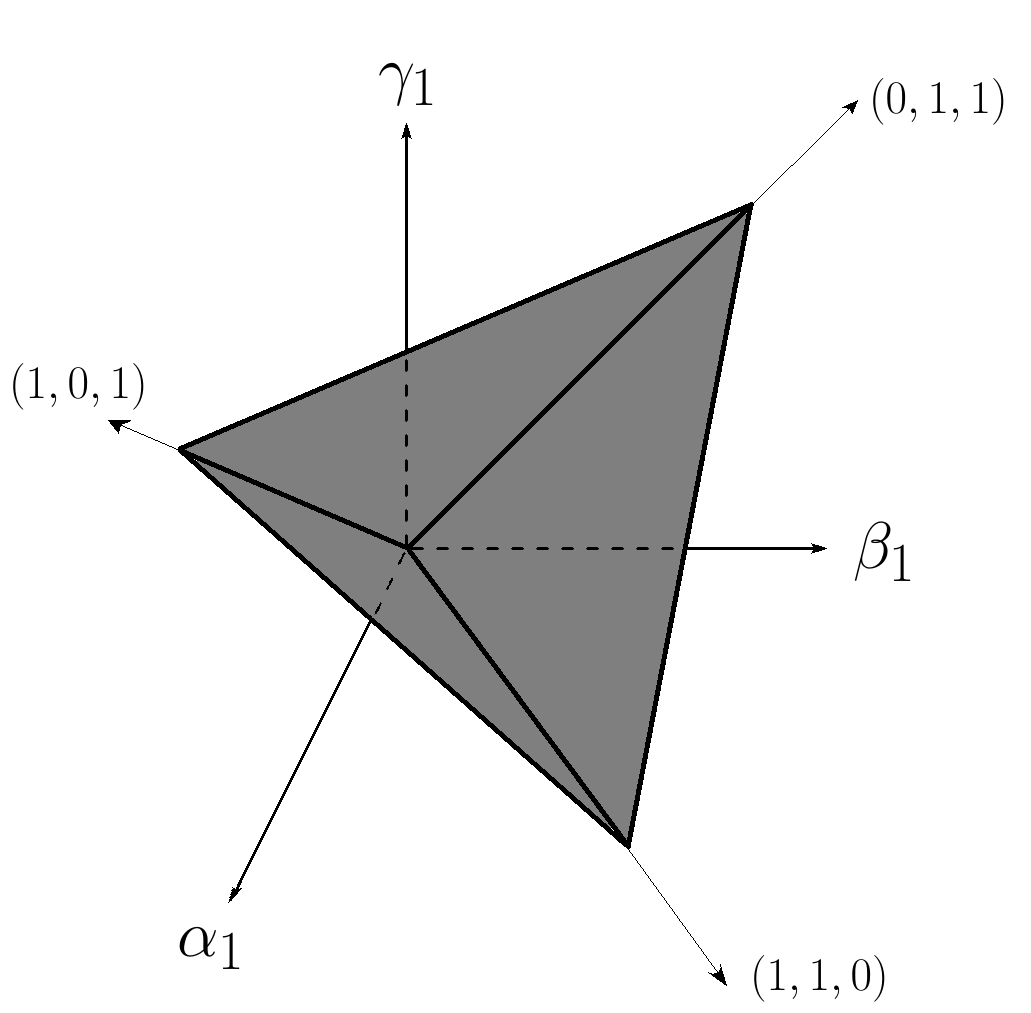}
\end{figure}
The valuation cone is the first quadrant.

\subsection{Tropical Compactification of the Maximal Torus of $\gl{2}$}\label{tropcompmaxtorus}
In this section we give an example of computing 
a tropical compactification in the spirit of Theorem~\ref{equivflat}, i.e.~by starting with the closure of $Y$ in a 
complete spherical variety and modifying it with blow-ups until the multiplication map of $\overline Y$
becomes flat. This amounts to refining the colored fan. The cones that lie outside $\trop Y$ are then  removed
and the multiplication map becomes flat and surjective. 
Let 
$$
T=\crb{x\in \gl{2}:x_{12}=x_{21}=0}.
$$
The tropicalization of $T$ is all of the valuation cone. Indeed, given a pair of integers $(\alpha_1,\alpha_2)$ with $\alpha_1\geq \alpha_2$, the invertible matrix
$$
\twomat{t^{\alpha_1}}{0}{0}{t^{\alpha_2}}
$$
has invariant factors $\alpha_1,\alpha_2$. We begin by compactifying $T$ in a spherical variety supported on $\cls{V}$, i.e. a complete spherical variety.
We view $\gl{2}$ as an open subset of $\aa^4$, with coordinates $x_{ij}$, which in turn is embedded in $\pp^4$, with homogeneous coordinates $(X_0,X)=(X_0,(X_{ij}))$, and is identified with $(X_0\neq 0)$,
see Example~\ref{tropcurvegl2}. 
The spherical variety $\pp^4$ has two closed $\gl{2}$-orbits: the origin of $0\in\aa^4$
 and the set of rank $1$ matrices ``at infinity''. 
 The colored fan of $\pp^4$ is shown in Figure \ref{wefvwefv} (A).

Let $T'\subset \pp^4$ be the closure of $T$. We claim that the multiplication map $\mu_{T'}:\gl{2}\times \gl{2}\times T'\rightarrow \pp^4$ is flat everywhere but $\mu_{T'}^{-1}(0)=\gl{2}\times \gl{2}\times \crb{0}$. We first show that all fibers of $\mu_{T'}$ but the one over $0\in \pp^4$ have the same dimension.
We~use the following simple observation:

\begin{prop} \label{fibersorbit}
Let $G$ be an algebraic group over $k$, $X$ a $G$-variety, and $Y\subseteq X$ a closed subvariety. The non-empty fibers of the multiplication map of $Y$
over points in an orbit $O$ have dimension $\dim G+\dim (Y\cap O)-\dim O$.
\end{prop}


The dimension of $\gl{2}\times \gl{2}$ is $8$, while the one of $\pp^4$ is $4$. We use Proposition \ref{fibersorbit} on each orbit of $\pp^4$ to show that the dimension of all fibers but the one over $0\in \pp^4$ is $6$. For each orbit $O$, we need to show that $\dim (T' \cap O)-\dim O=-2$.
\begin{enumerate}
\item If $O=\gl{2}$, then $T'\cap O=T$ is of dimension $2$, while $\dim O=4$.

\item The orbit $O$ of rank $1$ matrices in $\aa^4$, i.e. in $(X_0\neq 0)$, is the divisor $(\det x=0)\subset \aa^4$, without the origin $0$, hence of dimension $3$. It intersects $T'$ at the following subset
of dimension $1$:
$$
\twomat{x_{11}}{0}{0}{0},\quad \twomat{0}{0}{0}{x_{22}},\quad x_{11},x_{22}\in k^{\times}.
$$

\item Let $O$ be the orbit of invertible matrices at infinity. It is an open set in the hyperplane $(X_0=0)$, hence of dimension $3$. Its intersection with $T'$ 
$$
\brk{0,\twomat{X_{11}}{0}{0}{X_{22}}},\quad X_{11},X_{22}\in k^{\times},\, X_{11}X_{22}\neq 0.
$$
is isomorphic to $\pp^1\subset\pp^3$ minus two points, and so of dimension $1$.
\item Let $O$ be the orbit of rank $1$ matrices at infinity:
$$
O=\crb{(0,X)\in \pp^4:\det X=0}.
$$
It is a divisor on the hyperplane $(X_0=0)\cong \pp^3$, hence of dimension $2$. It intersects $T'$ at
a $0$-dimensional subset
$$
\brk{0,\twomat{1}{0}{0}{0}}\quad \text{and}\quad  \brk{0,\twomat{0}{0}{0}{1}}.
$$
\end{enumerate}

The fiber over $0$ is $\gl{2}\times \gl{2}\times \crb{0}$, which is of dimension $8$. We see that $\mu_{T'}$ is equidimensional everywhere but over the origin. Flatness of $\mu_{T'}$ over $\pp^4-\crb{0}$ follows from the following proposition, which is a direct consequence of \cite[Prop.~6.1.5]{EGAIV}. The closed set $T'$ is Cohen-Macaulay as a complete intersection in $\pp^4$, and $\gl{2}$ is an open set in $\aa^4$, thus $\gl{2}\times\gl{2}\times T'$ is Cohen-Macaulay.
\begin{prop} \label{flat615}
Let $\phi:X\rightarrow Y$ be a an equidimensional morphism of varieties such that
$Y$ is nonsingular and $X$ is Cohen-Macaulay.
Then $\phi$ is flat.
\end{prop}

\begin{figure}
    \centering
\caption{Fans of spherical varieties for $\gl{2}$}\label{wefvwefv}
    \begin{subfigure}[b]{0.45\textwidth}
        \includegraphics[width=\textwidth]{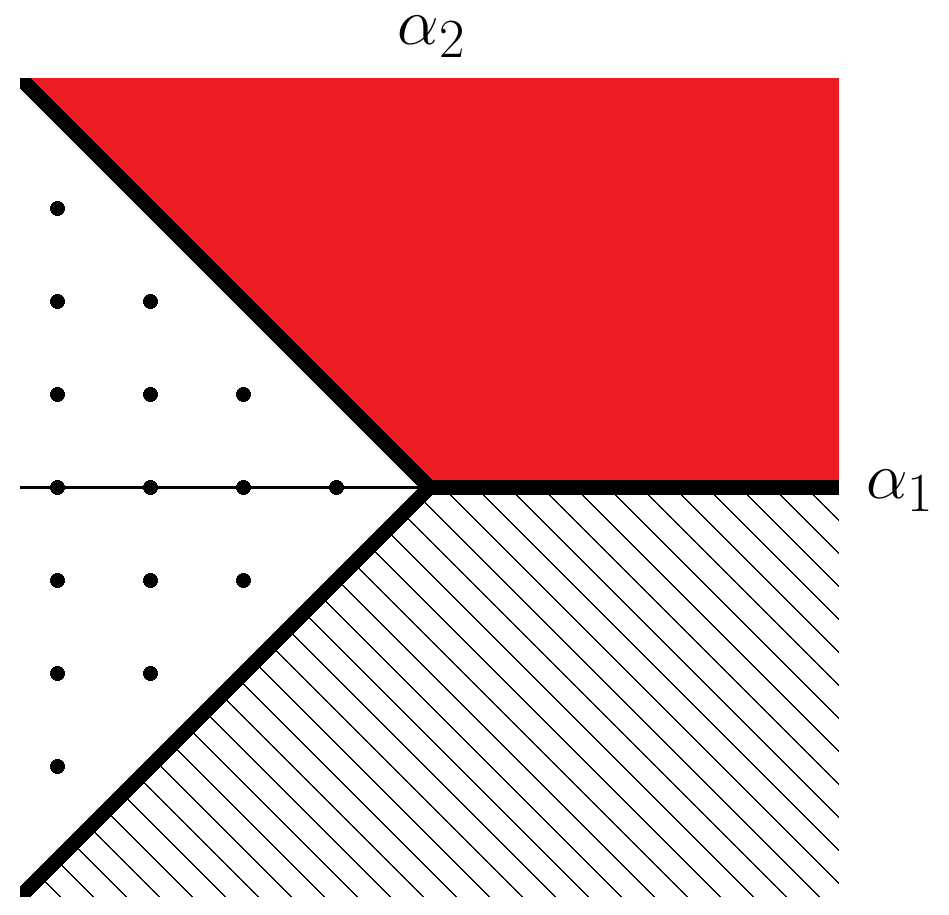}
        \caption{The fan of $\pp^4$}
        \label{fanproj4}
        
    \end{subfigure}
    \quad
    \begin{subfigure}[b]{0.45\textwidth}
        \includegraphics[width=\textwidth]{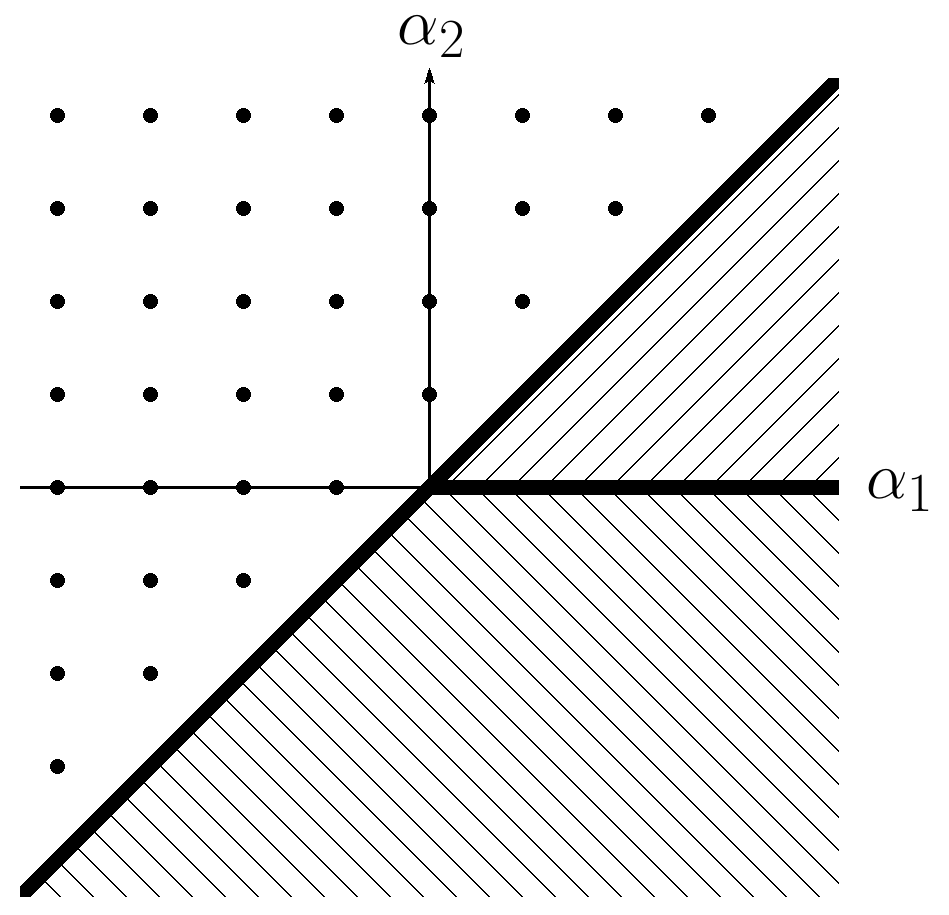}
        \caption{The fan of $\blow_0\pp^4$}
        \label{fanblowp4}
    \end{subfigure}
\end{figure}

Consider the blow-up morphism $\pi:\blow_0\pp^4\rightarrow \pp^4$ that restricts to an isomorphism $\blow_0\pp^4 -E\isom \pp^4-\crb{0}$. The exceptional divisor $E$ is isomorphic to $\pp^3$. We view its elements as $2\times 2$ homogeneous matrices; write $Y_{ij}$ for the associated homogeneous coordinates. The action of $\gl{2}\times\gl{2}$ on $\pp^4-\crb{0}$ extends to an action on $\blow_0\pp^4$ by left and right multiplication on the homogeneous matrices of the exceptional divisor. Under this action $\pi$ is a $\gl{2}$-morphism. Thus $\blow_0\pp^4$ is a spherical variety for the homogeneous space $\gl{2}$. The closed $\gl{2}$-orbits are the set of matrices of rank $1$ at infinity, and the matrices of rank $1$ in the exceptional divisor. The fan associated to $\blow_0\pp^4$ is given in Figure \ref{wefvwefv} (B). In particular $\blow_0\pp^4$ is toroidal.

We claim that the closure $\overline{T}\subset \blow_0\pp^4$ is a tropical compactification of $T$. Completeness of $\overline{T}$ follows from completeness of $\blow_0\pp^4$, or from the fact $\trop T=\supp \cla{F}$, where $\cla{F}$ is the fan associated to $\blow_0\pp^4$ (Prop.~\ref{propertrop}). Also, $\trop T=\supp \cla{F}$ implies that $\overline{T}$ intersects all orbits of $\blow_0\pp^4$ (Prop.~\ref{tevelevlem}), so that the multiplication map $\mu_{\overline{T}}:\gl{2}\times \gl{2}\times \overline{T}\rightarrow \blow_0\pp^4$ is surjective. We show that it is also flat.

The multiplication maps of $\overline{T}$ and $T'$ agree away from the exceptional divisor 
$$
\mu_{\overline{T}}|_{\gl{2}\times \gl{2}\times (\overline{T}-E)}=\mu_{T'}|_{\gl{2}\times \gl{2}\times (T'-\crb{0})}
$$
(as morphisms to $\blow_0\pp^4-E\cong \pp^4-\crb{0}$). The intersection $\overline{Y}\cap E$ consists of the diagonal homogeneous matrices of $E$. One can check using Proposition \ref{fibersorbit} that all fibers of $\mu_{\overline{T}}$ over $E$ are of dimension $6$; this case is identical to the case of fibers over the hyperplane $(X_0=0)$. Therefore all fibers of $\mu_{\overline{T}}$ are of the same dimension.

The closed set $\overline{T}$, which is the pure transform of $T'$, is Cohen-Macaulay as a complete intersection; this can be easily checked on the standard charts $U_{ij}\cong \pp^4$ of $\blow_0\pp^4\subset \pp^4\times \pp^3$. Applying Proposition \ref{flat615} we get that $\mu_{\overline{T}}$ is flat, and since it is surjective, faithfully flat. We deduce  $\overline{Y}\subset \blow_0\pp^4$ is a tropical compactification.

\begin{rem}\rm 
A similar argument shows that the closure of the maximal torus $T$ of a connected semi-simple 
group $G$ in its wonderful compactification $\overline G$, see \cite{CP}, 
is a spherical tropical compactification of $T$.
\end{rem}

\begin{rem}\rm
There are many spherical homogeneous spaces of rank~$1$ \cite[Tab.~5.10]{Ti}.
In this case a spherical toroidal tropical compactification of a subvariety $Y\subset G/H$ is 
uniquely determined by 
its spherical tropicalization $\trop(Y)$.
\end{rem}


\vspace{0.3cm}

{\small Department of Mathematics and Statistics,

University of Massachusetts,

Amherst, MA, USA.}

\vspace{0.3cm}


\begin{thebibliography}{99}



\bibitem[CP]{CP}
C.~de Concini, C.~Procesi, \emph{Complete symmetric varieties, I}, 
Lect. Notes in Math., \textbf{996}, pp. 1--44, 1983, Springer-Verlag

\bibitem[E]{E}
 D.~Eisenbud, \emph{Commutative Algebra with a View Toward Algebraic Geometry}, Graduate Texts in Mathematics, \textbf{150}, Springer, 1999

\bibitem[EGAIV]{EGAIV}
A.~Grothendieck, J. Dieudonn\'e,
\emph{El\'ements de g\'eom\'etrie alg\'ebrique: IV. \'Etude locale des sch\'emas et des morphismes de sch\'emas, Seconde partie},
Publ. Math. IH\'ES {\bf 24} (1965), 5--231.

\bibitem[EKL]{EKL}
M.~Einsiedler, M.~Kapranov, D.~Lind, \emph{Non-Archimedean amoebas and tropical varieties}, J. Reine Angew. Math. 601 (2006), 139--157.

\bibitem[F]{F}
W.~Fulton, \emph{Eigenvalues, invariant factors, highest weights, and Schubert calculus}, Bulletin of the
Amer. Math. Society, Vol. 37, Number 3 (2000), pp. 209--249

\bibitem[G]{G}
W.~Gubler, \emph{A guide to tropicalizations}, Algebraic and Combinatorial Aspects of Tropical Geometry, Contemporary Mathematics, Vol. 589, Amer. Math. Soc., Providence, RI, 2013, pp. 125--189

\bibitem[H]{H}
R.~Hartshorne, \emph{Algebraic Geometry}, 
Springer-Verlag, 1977

\bibitem[HKT]{HKT}
P.~Hacking, S.~Keel, J.~Tevelev, \emph{Stable Pair, Tropical, and Log Canonical Compact Moduli of Del Pezzo Surfaces}, Inventiones {\bf 178}, no.1 (2009), 173-228


\bibitem[K]{K}
F.~Knop, \emph{The Luna-Vust theory of spherical embeddings}, Proc. Hyderabad Conf. on
Algebraic Groups, December 1989, Madras: Manoj Prakashan (1991), 225-249.


\bibitem[LM]{LM}
G.~Laumon, L.~Moret-Bailly, \emph{Champs alg\'{e}briques}, Ergebnisse der Mathematik und ihrer Grenzgebiete. 3. Folge., vol. 39, Springer-Verlag, Berlin, 2000.

\bibitem[LQ]{LQ}
M.~Luxton, Z.~Qu, \emph{Results on tropical compactifications}, Trans. AMS. \textbf{363} (2011), 4853--4876

\bibitem[LV]{LV}
D.~Luna, T.~Vust, \emph{Plongements d'espaces homog\`{e}nes},  Com. Math. Helv. 58 (1983), 186--245


\bibitem[MB]{MB}
D.~Maclagan, B.~Sturmfels, \emph{Introduction to tropical geometry}, American Mathematical Society, Graduate Studies in Mathematics, \textbf{161}, 2015.


\bibitem[P]{P}
S.~Payne, \emph{Analytification is the limit of tropicalizations}, Math. Res. Let. \textbf{16}-3 (2009), 543--556 

\bibitem[R]{R}
M.~Raynaud, \emph{Flat modules in algebraic geometry}, Comp. Math., 24, no 1 (1972), pp. 11--31

\bibitem[RG]{RG}
M.~Raynaud, L. ~Gruson, \emph{Crit\`{e}res de platitude et de projectivit\'{e}. Techniques de ``platification'' d' un module}, Invent. Math. 13 (1971), pp. 1--89

\bibitem[ST]{ST}
B.~Sturmfels, J.~Tevelev, \emph{Elimination theory for tropical varieties}, Math. Res. Let., {\bf 15} (2008)

\bibitem[Su1]{Suv}
H.~Sumihiro, \emph{Equivarian completion}, Jour. of Math. Kyoto Univ. Vol 14 (1974), 1--28

\bibitem[Su2]{Su}
H.~Sumihiro, \emph{Equivarian completion II}, Jour. of Math. Kyoto Univ. Vol 15 (1975), 573--605

\bibitem[TDTE]{TDTE}
A.~Grothendieck, \emph{Technique de descente et th\'{e}or\`{e}mes d'existence en g\'{e}om\'{e}trie alg\'{e}brique: les sch\'{e}mas de Hilbert}, S´eminaire Bourbaki, {\bf 221}, 1960/61

\bibitem[Te]{Te}
J.~Tevelev, \emph{Compactifications of subvarieties of tori}, Amer. J. Math, {\bf129}-4 (2007), 1087--1104


\bibitem[Th]{Th}
R.~W.~Thomason, \emph{Algebraic $K$-theory of group scheme actions}, Algebraic Topology
and Algebraic K-theory, Ann. of Math. Stud. {\bf 113} (1987), 539--563.

\bibitem[Ti]{Ti}
D.~Timashev, \emph{Homogeneous Spaces and Equivariant Embeddings}, Springer, Encycl. of Math. Sciences, Vol. 138 (2011)

\bibitem[U]{U}
M.~Ulirsch, \emph{Tropical compactification in log-regular varieties}, Math. Zeit. (2015), {\bf 280},  (1--2), 
pp. 195--210

\bibitem[V]{V}
E.~Vinberg, \emph{On invariants of a set of matrices}, 
J. Lie Theory, 6(2) (1996), pp.~249--269

\end{thebibliography}
\end{document}